\documentclass[a4paper]{amsart}
\usepackage{a4}

\usepackage{amsmath}
\usepackage{amssymb}
\usepackage[utf8]{inputenc}
\usepackage[T1]{fontenc}
\usepackage{graphicx}
\usepackage{subcaption}
\usepackage{enumerate}
\usepackage{color}
\usepackage{float}
\usepackage{booktabs}
\usepackage{multirow}
\usepackage{algorithm}
\usepackage{algpseudocode}
\usepackage{mathrsfs}
\usepackage{xcolor}
\usepackage{booktabs}

\def\D{\Delta}

\newcommand{\cF}{{\mathcal F}}

\newcommand{\ds}{\displaystyle}

\newcommand{\diver}{{\rm{div}}}
\newcommand{\T}{{\mathbb{T}}}

\newtheorem{remark}{\textbf{Remark}}[section]

\newtheorem{proposition}{\textbf{Proposition}}[section]

\numberwithin{equation}{section}

\title[Newton's iterations for MFG]{Newton Methods for Mean Field Games: A Numerical Study}

\author[Elisabetta Carlini]{Elisabetta Carlini\cs}
\thanks{\cs Dipartimento di Matematica Guido Castelnuovo, Sapienza Universit\`a di Roma, 00185 Rome, Italy (carlini@mat.uniroma1.it).}
\author[Ahmad Zorkot]{Ahmad Zorkot\ha}
\thanks{\ha Faculty of Mathematics, University of Vienna, Oskar-Morgenstern-Platz 1, 1090 Vienna, Austria. (ahmad.zorkot@univie.ac.at).}
\newcommand{\cs}{$^1$} \newcommand{\ha}{$^2$}

\def\dd{{\rm d}}

\def\weight(#1,#2){c_{#1,#2}}

 \if {
  
 } \fi

\if {

}{{\caption}}
} \fi

\def\Dt{\Delta t}

\def\B{\mathcal{B}}

\def\G{\mathcal{G}}

\def\I{\mathcal{I}}

\def\eps{\varepsilon}

\def\diag{{\mathop{\rm diag}}}

\def\1B{{\bf  1}}

\newcommand{\NN}{\mathbb{N}}

\newcommand{\RR}{\mathbb{R}}

\newcommand\be{\begin{equation}}
\newcommand\ee{\end{equation}}
\newcommand\ba{\begin{array}}
\newcommand\ea{\end{array}}
\newcommand{\bean}{\begin{eqnarray*}}
\newcommand{\eean}{\end{eqnarray*}}

\def\ds{\displaystyle}

\begin{document}

\maketitle

\begin{abstract}
  We address the numerical solution of second-order Mean Field Game problems through Newton iterations in infinite dimensions, introduced in \cite{newton}, where quadratic convergence of the method was rigorously established. Building upon this theoretical framework, we develop new numerical discretization techniques, including both a finite difference and a semi-Lagrangian scheme, that enable an effective computational implementation of the infinite-dimensional iterations.
  The proposed methods are tested on several benchmark problems, and the resulting numerical experiments demonstrate their robustness, accuracy, and efficiency. A comparative analysis between the two schemes and existing approaches from the literature is also presented, highlighting the potential of Newton-based solvers for MFG systems.
\end{abstract}

\section{Introduction}
Mean field games (MFG) offer a powerful framework for modeling the strategic behavior of a large
number of indistinguishable agents who interact through the statistical distribution of their states. Introduced independently by Lasry and Lions \cite{ll2,ll3,L-L} and by Huang, Malhamé, and Caines \cite{malhame2,malhame4}, MFG has become a central tool for the analysis of Nash equilibria in stochastic differential games with infinitely many players. MFG models have found applications in a broad range of fields, including economics, finance, traffic flow, crowd dynamics, and the social sciences. We refer to \cite{carmonafinance,ext6,mfgnash,mfgnash2,MR3195844} for general overviews of the theory and its applications.

From an analytical viewpoint, the classical formulation of MFGs introduced in \cite{ll2,ll3} consists of a coupled system of partial differential equations (PDEs). The backward component is a Hamilton--Jacobi--Bellman (HJB) equation associated with the optimal control problem of a representative agent, complemented by a terminal condition, while the forward component is a Fokker--Planck (FP) equation governing the evolution of the population density and subject to an initial condition. Comprehensive presentations of this framework can be found in the monographs~\cite{ext7,ext6,mfgnash2}, the survey~\cite{MR3195844}, and the lecture notes~\cite{MR4214773}.

In this work, we focus on the following time-dependent, second-order Mean Field Game system with local coupling, posed on the $d$-dimensional torus $\mathbb{T}^d$:
\begin{equation}
\label{MFGg}
\begin{cases}
-\partial_tu-\nu\Delta u+H(x,Du)=F(m(t,x)) & \text{in }[0,T]\times\mathbb T^d,\\
\partial_tm -\nu\Delta m-\diver(mD_pH(x,Du))=0 & \text{in }[0,T]\times\mathbb T^d,\\
m(0,x)=m_0(x), u(T,x)=G(x) & \text{in }\mathbb T^d,
\end{cases}
\end{equation}
where $T > 0$ denotes the time horizon, $\nu > 0$ is the diffusion coefficient, $H$ is a convex Hamiltonian, and $F$ is a coupling term depending locally on the density $m$. The choice of the torus $\mathbb{T}^d=\mathbb{R}^d/\mathbb{Z}^d$ allows us to avoid the treatment of boundary conditions.

Assume that the following hypotheses hold:
\begin{enumerate}
  \item[(i)] the initial density satisfies $m_0 \in L^\infty(\mathbb{T}^d)$, $m_0 \geq 0$, and
  $\int_{\mathbb{T}^d} m_0(x)\,dx = 1;  $
  \item[(ii)] the Hamiltonian $H(x,p)$ is such that $H(x,p)$ and its derivative $H_p(x,p)$ are of class $C^1$, moreover, $H(x,p)$ is convex with respect to the momentum variable $p$;
  \item[(iii)] the Hamiltonian satisfies a global Lipschitz condition in $p$, i.e. there exists a constant $\beta>0$ such that
  $$
  |H_p(x,p)| \leq \beta \qquad \forall (x,p) \in \mathbb{T}^d \times \mathbb{R}^d;
  $$
  \item[(iv)] the coupling functions $F(m)$ and $G(m)$ are continuous and bounded from below.
\end{enumerate}
Under these assumptions, there exists a classical solution $(u,m)$ to the MFG system~\eqref{MFGg} such that $(u,m)$ belongs to the parabolic Hölder space $C^{2+\alpha,\,1+\alpha/2}(Q)$ for some $\alpha \in (0,1)$. 
Furthermore, if the coupling functions $F(\cdot)$ and $G(\cdot)$ are non-decreasing with respect to the density $m$, the solution is unique; see~\cite[Theorem~1.11]{CardPorr2020}.
In addition, the density $m(t,\cdot)$ remains a probability measure on the unit torus $\mathbb{T}^d$ for all $t \in [0,T]$.

The numerical approximation of MFG systems has been the subject of intensive research in recent years. For second-order (stochastic) problems such as \eqref{MFGg}, finite difference schemes have been investigated in \cite{finitedifference1,finitediff2,MR3452251}, while semi-Lagrangian approaches have been developed in \cite{CSorder2,Calzola_Carlini_Silva_lg_second_order}. More recently, data-driven techniques, including deep learning, reinforcement learning, and tensor-based methods, have been applied to MFG models in \cite{MR4264647,MR4522347,MR4440805,carlinisaluzzi2025}. In the first-order (deterministic) setting, we refer to \cite{CSorder1,carlini2023lagrangegalerkin,saeed,gianatti2023approximation,GS,carlinicoscetti2025}. Broader surveys of numerical methods for MFGs can be found in \cite{mathieuachdou,mathieu1} and the references therein.

The main focus of the present work is the application of Newton’s method in infinite-dimensional function spaces to the continuous MFG system \eqref{MFGg}. This approach, originally proposed in \cite{newton} (see also \cite{berry-silva} for the stationary case), reformulates the coupled PDE system as a nonlinear operator equation and constructs successive approximations by linearization around the current iterate. Each Newton step therefore requires the solution of a coupled forward--backward linear parabolic system. A key theoretical result established in \cite{newton} is the quadratic local convergence of the continuous Newton iterations under suitable regularity and proximity assumptions.

For clarity of presentation, we introduce the discretization in the case of a separable Hamiltonian. Nevertheless, the proposed methodology naturally extends to non-separable Hamiltonians, and this extension is illustrated through a dedicated numerical experiment.

To implement Newton’s method for system \eqref{MFGg}, we introduce the operator
\be
\label{function-F}
\cF:(u,m)\to
\left( 
\begin{array}{l}
	-\partial_tu-\nu\Delta u+H(\cdot,Du)-F(m)\\
	\partial_tm -\nu\Delta  m-\diver(mD_pH(\cdot,Du))\\
	u(T,\cdot)-G(\cdot)\\
	m(0,\cdot)-m_0(\cdot)
\end{array}
\right),
\ee
so that system \eqref{MFGg} can be written compactly as
\[
\mathcal{F}(u,m)=0.
\]

Assuming sufficient smoothness of the data, the Newton iterates $(u^n,m^n)$ for $n\geq 1$ are defined by
\begin{equation}
\label{eq:Newton_iter}
	J\cF(u^{n-1},m^{n-1})(u^{n}-u^{n-1},m^{n}-m^{n-1})=- \cF(u^{n-1},m^{n-1}),
\end{equation}
where $J\cF$ denotes the Jacobian of $\cF$, given explicitly by
\be
\label{jacob}
J\cF(u,m)(v,\rho)=\left( 
\begin{array}{lr}
	-\partial_t v -\nu\Delta v+D_pH(\cdot, Du)D v\quad &-F'(m)\rho\\[4pt]
	-\diver (mD_{pp}H(\cdot, Du)D v)                      	&\partial_t \rho -\nu\Delta  \rho-\diver(D_pH(\cdot, Du)\rho)\\
	v(T,\cdot)&0\\
	0&\rho(0,\cdot) 
\end{array}
\right),
\ee
for H\"older continuous functions $v$ and $\rho$ defined on $[0,T]\times\mathbb{R}^d$.

Setting $q^{n}=D_pH(\cdot,Du^{n-1})$ and substituting \eqref{jacob} into \eqref{eq:Newton_iter}, we obtain the linearized system
\begin{equation}
\label{eq:Newton}
\begin{cases}
    -\partial_t u^{n} - \nu \Delta u^{n} + q^{n} Du^{n} = q^{n} Du^{n-1} - H(x,Du^{n-1})+ F(m^{n-1}) \\
    \hspace{5cm} + F'(m^{n-1})(m^{n}-m^{n-1})\quad &\text{in }[0,T]\times\mathbb T^d, \\[6pt]
    \partial_t m^{n} - \nu\Delta m^{n} - \diver (m^{n} q^{n}) = \diver (m^{n-1}D_{pp}H(x,Du^{n-1})(Du^{n}-Du^{n-1}))\quad &\text{in }[0,T]\times\mathbb T^d, \\[6pt]
    m^{n}(x,0) = m_0(x),\quad u^{n}(x,T)=G(x)\quad &\text{in }  \mathbb T^d. 
\end{cases}
\end{equation}

It was shown in \cite{newton} that, under appropriate assumptions, the sequence generated by these Newton iterations converges locally to the unique solution of \eqref{MFGg}, with a quadratic rate provided the initial guess is sufficiently close to the exact solution. Despite these strong theoretical guarantees, the numerical behavior of the infinite-dimensional Newton method has not been extensively investigated. The aim of the present work is to address this gap by designing suitable discretization schemes and assessing their performance through numerical experiments.

To discretize and solve the linearized systems \eqref{eq:Newton}, we consider two complementary numerical strategies: a semi-Lagrangian (SL) method and an implicit finite difference scheme. For both approaches, we establish well-posedness results and evaluate their efficiency and accuracy numerically. We also compare our methodology with the Newton-based finite difference scheme introduced in \cite{finitedifference1}, highlighting the distinction between the paradigms of ``iterate then discretize'', where Newton’s method is applied at the continuous level, and ``discretize then iterate'', where Newton’s method is applied directly to the discrete nonlinear system. Our experiments indicate that both strategies yield comparable accuracy and computational cost, while the former offers a simpler and more flexible discretization framework.

In contrast to approaches such as \cite{finitedifference1,CSorder2}, which discretize the fully nonlinear MFG system \eqref{MFGg} and subsequently apply Newton’s method at the discrete level, our strategy applies standard finite difference and semi-Lagrangian discretizations only to the linear systems arising at each Newton iteration. This results in a significantly simpler numerical structure. Moreover, our results show that the semi-Lagrangian scheme exhibits greater robustness in the hyperbolic regime, corresponding to small values of $\nu$. Finite-difference-based schemes, on the other hand, may experience numerical instabilities in this setting, in agreement with known limitations discussed in \cite{achdou1,achdou2}, and often require continuation strategies to gradually decrease $\nu$.

Since Newton’s method is only locally convergent, its effectiveness strongly depends on the choice of the initial guess. To enhance robustness when standard Newton iterations fail, we incorporate a globalized Newton strategy based on a line-search procedure, following \cite{global-newton}. This globalization technique, driven by a suitable merit function, proves particularly effective for the finite-difference-based Newton solver in the hyperbolic regime.

The remainder of the paper is organized as follows. Section~\ref{SL_newton} is devoted to the construction of a semi-Lagrangian discretization of the linearized system. In Section~\ref{upwind}, we introduce an implicit finite difference scheme. Finally, Section~\ref{numerics} presents the numerical experiments and a comparative analysis of the proposed methods.

\section{A semi-Lagrangian scheme}
\label{SL_newton}
In this section, we discretize the iterative system  \eqref{eq:Newton} by means of a semi-Lagrangian scheme in the two dimensional state-space and we prove the well-posedness of the discrete system.\\ 
We refer to \cite{Calzola_Carlini_Silva_lg_second_order,CSorder2} for the early work on approximating the second-order MFG systems using a semi-Lagrangian scheme, and to \cite{CC} for a semi-Lagrangian scheme applied to parabolic equations. 

\subsection{Notations and definitions}\label{sec:def}
In what follows, to simplify the discussion, we take $d = 2$ and consider the Hamiltonian
\begin{equation}
  H(x,p) = \frac{|p|^2}{2} -V(x), \qquad \text{for } (x,p) \in \mathbb{T}^2 \times \mathbb{R}^2.
\end{equation}
where $V$ is a given bounded potential.

Given two positive integers $N_t$ and $N_h$, we define the time step $\Delta t = \frac{T}{N_t}$ and the space step $h = \frac{1}{N_h}$. 
We also introduce the index sets 
$$
\I_{\Delta t} := \{0, \dots, N_t\}, \qquad 
\I_{\Delta t}^* := \I_{\Delta t} \setminus \{N_t\}, \qquad 
\I_h := \{0, \dots, N_h-1\}.
$$
We define the discrete time grid $\mathcal G_{\Delta t} := \{t_k = k\Delta t \mid k \in\I_{\D t}\}$ and discrete space grid $\mathcal G_{h} := \{x_{i,j} = (ih, jh) \mid i,j \in \I_{h} \}$. 
We denote by $B(\G_h)$ and $B(\G_{\D t}\times\G_h)$ the two sets of functions defined in $\G_h$ and $\G_{\D t}\times\G_h$ respectively.

The objective is to construct two discrete functions 
\(u^{n,k}_{[i,j]}\) and \(m^{n,k}_{[i,j]}\) approximating the solution \((u^n,m^n)\) of \eqref{eq:Newton} at the grid points \((t_k,x_{i,j})\) for all \(k\in\I_{\Delta t}\) and \(i,j\in\I_h\). 
The index operator \([\cdot,\cdot]\), defined by \([i,j]=\big((i+N_h)\bmod N_h,\,(j+N_h)\bmod N_h\big)\), accounts for periodic boundary conditions. 
For notational simplicity we will also write \((u^{n,k}_{i,j},m^{n,k}_{i,j})=(u^{n,k}_{[i,j]},m^{n,k}_{[i,j]})\).
Given a grid function $v\in\B(\G_h)$, we introduce the first order central differences operators
\be\label{eq:Dx-Dy}
\ba{rcl}
\ds
 (D_1 v)_{i,j}&=\ds\frac{v_{i+1,j}-v_{i-1,j}}{2h}\quad \text{for }i,j\in\I_{h},\\[10pt]
\ds (D_ 2v)_{i,j}&=\ds\frac{v_{i,j+1}-v_{i,j-1}}{2h}\quad \text{for }i,j\in\I_{h},
\ea
\ee
and define the discrete gradient operator $D_h$ as 
\be
\label{gradient_fd}
(D_h v)_{i,j}=((D_1 v)_{i,j},(D_2  v)_{i,j})^{\top}\quad \text{for }i,j\in\I_{h}.
\ee
Then, for a discrete vector field $p = (p_1, p_2) \in B(\G_h)^2$, 
we define the discrete divergence operator $\diver_h$ as
\begin{equation}
\label{div_fd}
(\diver_h(vp))_{i,j}
= \frac{1}{2h} \Big(
v_{i+1,j}(p_1)_{i+1,j} - v_{i-1,j}(p_1)_{i-1,j}
+ v_{i,j+1}(p_2)_{i,j+1} - v_{i,j-1}(p_2)_{i,j-1}
\Big),
\quad i,j \in \I_h.
\end{equation}

Given $\phi\in \B(\G_{h})$, we define its piecewise linear interpolant  as
$$
I[\phi](x)=\sum_{i,j\in\I_{h}}\beta_{i,j}(x)\phi_{i,j}\quad\text{for all } x\in\mathbb T^2,
$$
where $\phi_{i,j}=\phi(x_{i,j})$, and $\beta_{i,j}$ denote the piecewise linear finite element basis associated to the grid $\G_h$. 
For any function $\varphi:\mathbb T^{2}\to\mathbb R$, let $\varphi|_{\mathcal G_h}$ denote its restriction to the grid $\mathcal G_h$. If $\varphi\in C^{2}(\mathbb T^{2})$ and its second-order derivatives are bounded, then, by \cite[Remark~3.4.2]{quarteronivalli94}, the interpolation error satisfies
\begin{equation}
\label{interpolation_phi}
\|\varphi - I[\varphi|_{\mathcal G_h}]\|_{\infty}\le Ch^{2},
\end{equation}
where $C>0$ depends only on $\varphi$.

\subsection{Semi-Lagrangian scheme for the backward equation}
Building on the notations and framework introduced in the previous section, we now apply the semi--Lagrangian discretization to approximate the iterative system~\eqref{eq:Newton}.
We start with the backward equation, which can be written as follows:
\begin{equation}
\label{first_eq}
\begin{cases}
    -\partial_t u^{n} - \frac{\sigma^2}{2} \Delta u^{n} + q^{n} Du^{n} - G^n(t,x)
    =0\quad&\text{in } [0,T]\times\mathbb T^2,\\
   u^{n}(x,T)=G(x) \quad&\text{in }\mathbb T^2, 
\end{cases}
\end{equation}
where $$\sigma=\sqrt{2\nu}, \quad q^n(t,x)=Du^{n-1}(t,x)$$ and
\begin{equation*}
\ba{rcl}
G^n(t,x)&=&\ds\frac{|q^n(t,x)|^2}{2}+V(x)\\[6pt]
&\hspace{0.2cm}&+\ds F(m^{n-1}(t,x)) 
+F'(m^{n-1}(t,x))(m^{n}(t,x)-m^{n-1}(t,x)),\quad \text{for }(t,x)\in[0,T]\times\mathbb T^2.
\ea
\end{equation*}
By the Feynman-Kac formula (see e.g \cite{yong}), the solution $u^n$ to \eqref{first_eq}, admits the following representation
\be
\label{exp}
u^n(t,x)=\mathbb E\Big[\int_t^T G^n(s,X^{t,x}(s))\dd s+G(X^{t,x}(T))\Big],\quad \text{for }(t,x)\in[0,T]\times \mathbb T^2,
\ee
where $X^{t,x}$ denotes characteristics solving
\begin{equation}
\label{f_eq}
\left\{
\begin{aligned}
   dX(s)&=q^n(s,X(s))+\sigma dW(s)\quad \text{for } s\in(t,T),\\
   X(t)&=x .
\end{aligned}
\right.
\end{equation}
We explain now how to construct a SL approximation using the technique shown in \cite{falcone}. 
First, notice that \eqref{exp} imply that for every $k\in\I^*_{\D t}$, we have
\be
\label{exp2}
u^n(t_k,x)=\mathbb E\Big[\int_{t_k}^{t_{k+1}}G^n(s,X^{t_k,x}(s) )\dd s+u^{n}(t_{k+1}, X^{t_k,x}(\Delta t) )\Big].
\ee
We denote by
$q^{n,k}_{i,j}\in\RR^{2}$ an appropriate approximation of the drift term $q^n(t_k,x_{i,j})$ for each $k\in\I_{\Dt}$ and $i,j\in\I_{h}$. 
Then, we approximate the expectation in \eqref{exp2} (see e.g \cite{Kloeden}) as
\be
\label{exp3}
\mathbb E\left[u^n(t_k,  X^{t_k,x_{i,j}}(\Delta t))\right] = \frac{1}{4}\sum_{\ell=1}^4 u^n(t_k, y^{\ell}(x_{i,j},q^{n,k}_{i,j})) + O((\Delta t)^2),
\ee
where
\be
\label{y_k}
y^{\ell}(x_{i,j},q^{n,k}_{i,j})=(x_{i,j}+\Dt q^{n,k}_{i,j}+\sqrt{2\D t}\sigma e^\ell)_p\quad \text{for }\ell=1,\hdots,4,
\ee
with $e^\ell$ representing four vectors of $\RR^2$ with one component equal to $\pm 1$ and the other null, and
$$
(z)_p =(z_1 - \lfloor z_1 \rfloor, z_2 - \lfloor z_2 \rfloor)\quad \text{for all }
z=(z_1,z_2)\in\RR^2,$$
denotes the periodic projection on $\mathbb T^2$ of $z\in\RR^2$.

Finally, by combining \eqref{exp}, \eqref{exp2}, and \eqref{exp3}, and using the rectangular rule to approximate the integral term in \eqref{exp2},
we define the semi--Lagrangian scheme for equation~\eqref{first_eq} as follows.
Given $q^n\in B(\G_{\D t}\times\G_h)^2$ and $m^{n},m^{n-1}\in B(\G_{\D t}\times\G_h)$,
find $u^{n}\in B(\G_{\D t}\times\G_h)$  such that
\begin{equation}
\label{scheme-B}
\begin{cases}
\begin{split}
  u^{n,k}_{i,j}&=S^n_{k,(i,j)}(u^{n,k+1})\quad &\text{for all } k\in \mathcal I_{\Delta t}^*, \, i,j\in\I_{h}, \\[4pt]
   u^{n,N_{t}}&= G(x_{i,j}), \\[4pt]
\end{split}
\end{cases}
\end{equation}
where, for every $f\in B(\G_h)$, $k\in \mathcal I_{\Delta t}^*$, and $i,j\in\I_h$,
\begin{equation*}
\label{SLscheme}
\begin{split}
S^n_{k,(i,j)}(f):=\frac{1}{4}  \sum_{\ell=1}^4 I[f](y^{\ell}(x_{i,j},q^{n,k}_{i,j}))
+\D t G^n( t_k,x_{i,j}).
\end{split}
\end{equation*}
Let $\phi \in C^2_0(\mathbb{T}^2)$, and let us denote by $C$ a positive real number which can depend only on $\phi$. From \eqref{y_k} and Taylor expansion (see e.g \cite{CC}) we have
\be
\label{const}
\Big|\frac{1}{4}\sum_{\ell=1}^4
 \phi(y^\ell (x_{i,j},q^{n,k}_{i,j})) - \left( \phi(x_{i,j}) + \Delta t \frac{\sigma^2}{2} \Delta \phi(x_{i,j}) + \D t q^{n,k}_{i,j}D\phi(x_{i,j}) \right)\Big| 
 \leq C(\Delta t)^2 \quad\text{for all } i,j\in\I_h.
\ee
Let $u^n$ be a smooth solution to \eqref{first_eq}, with bounded derivatives. 
Then, using \eqref{interpolation_phi} and \eqref{const}, for every $k \in \I_{\Delta t}^*$ and $i,j \in \I_h$, 
the consistency error of scheme~\eqref{scheme-B} satisfies
$$
T_{\Delta t, h}(t_k, x_{i,j}) := \frac{1}{\Delta t} \left(u^n(t_{k+1}, x_{i,j}) - S^n_{k,(i,j)}(u^n(t_{k+1})\right)=O\Big({\D t}+\frac{h^2}{\D t}\Big)\quad (t,x)\in(0,T)\times\mathbb T^2.
$$

\subsection{A semi-Lagrangian scheme for the forward equation}
Let us now consider the second equation in system \eqref{eq:Newton}
\begin{equation}
\label{second_eq}
\begin{cases}
\partial_tm^n-\frac{\sigma^2}{2}\Delta m^n-\diver(m^nq^n)=\diver(m^{n-1}(t,x)(Du^{n}(t,x)-Du^{n-1}(t,x)))
&\text{in }[0,T]\times\mathbb T^2,\\
m^n(0,x)=m_0(x) &\text{in }\mathbb T^2.
\end{cases}
\end{equation}
Following the same analogue in \cite{CSorder2},
we propose the following scheme to approximate \eqref{second_eq}.
Given $q^n,D_hu^{n-1},D_hu^n\in B(\G_{\D t}\times\G_h)^2$, and $m^{n-1}\in B(\G_{\D t}\times\G_h)$,
find $m^{n}\in B(\G_{\D t}\times\G_h)$  such that
\begin{equation}
\label{scheme-F}
\begin{cases}
\begin{split}
  m^{n,k+1}_{i,j}&=(S^n_{k,(i,j)})^*(m^{n,k})\quad \text{for all } k\in \I^*_{\Delta t},\, i,j\in\I_{h}, \\[4pt]
   m^{n,0}_{i,j}&= m_0(x_{i,j}),
\end{split}
\end{cases}
\end{equation}
 where, for a given function $f\in B(\G_h)$, $k\in\I_{\D t}^*$ and $i,j\in\I_h$
\begin{equation*}
\label{Lscheme}
\begin{split}
(S^n_{k,(i,j)})^*(f)=\frac{1}{4}\sum_{\ell=1}^4 I^*[f](y^{\ell}(x_{i,j},q^{n,k}_{i,j}))+\Delta t {(\diver_h(m^{n-1,k+1}({D}_hu^{n,k+1}-D_hu^{n-1,k+1})))_{i,j}},
\end{split}
 \end{equation*}
 where, for every $i,j\in\I_h$, $I^*[f](y^{\ell}(x_{i,j},q^{n,k}_{i,j}))$ is the adjoint operator of $f\to I[f](y^{\ell}(x_{i,j},q^{n,k}_{i,j}))$.
 
 Let $\phi, \psi \in C^4(\mathbb{T}^2),\; \varphi \in C^3(\mathbb{T}^2)$,
and set $p = D\phi - D\psi$. Using the definitions of the discrete gradient and divergence \eqref{gradient_fd}--\eqref{div_fd}, we have
\[
\mathrm{div}_h\big(\varphi (D_h \phi - D_h \psi)\big)_{i,j} 
= \frac{1}{2h} \Big( \varphi_{i+1,j} p_{1,i+1,j} - \varphi_{i-1,j} p_{1,i-1,j} + \varphi_{i,j+1} p_{2,i,j+1} - \varphi_{i,j-1} p_{2,i,j-1} \Big).
\]

Defining the smooth function $w = \varphi p$, a Taylor expansion at $x_{i,j}$ yields
\[
\frac{w_1(x_{i+1,j})-w_1(x_{i-1,j})}{2h} = \partial_x w_1(x_{i,j}) + O(h^2), 
\quad
\frac{w_2(x_{i,j+1})-w_2(x_{i,j-1})}{2h} = \partial_y w_2(x_{i,j}) + O(h^2),
\]
and hence
\[
\mathrm{div}_h(\varphi p)_{i,j} = \mathrm{div}(\varphi p)(x_{i,j}) + O(h^2).
\]

Therefore, under these regularity assumptions, the discrete term 
\[
\mathrm{div}_h\big(\varphi (D_h \phi - D_h \psi)\big)
\]
is a second-order consistent approximation 
\[
\mathrm{div}\big(\varphi (D \phi - D \psi)\big),
\]
showing that the source term \eqref{Lscheme} is $O(h^2)$ consistent in space.

\subsection{The discrete Newton system }\label{sec:Newton}
Given $v \in \B(\G_h)$, we denote by $V\in\RR^{{N_h^2}}$ the vectors such that
\begin{equation}
V_{i + j N_h} = v_{i,j}\quad \forall \quad i,j\in \I_h.
\end{equation}
Then the discrete gradient operator $D_h : \mathbb{R}^{N_h} \to \mathbb{R}^{2N_h}$, defined in \eqref{gradient_fd}, is such that \begin{equation}
D_h V = 
\begin{pmatrix} D_1 \\ D_2 \end{pmatrix} V =
\begin{pmatrix} D_1 V \\ D_2 V \end{pmatrix},
\end{equation}
where the matrices $D_1, D_2 \in \mathbb{R}^{N_h\times N_h}$ corresponds  to the operator defined in \eqref{eq:Dx-Dy}.

For a discrete vector field $p = (p_1,p_2) \in \mathbb{R}^{2N_h}$, the discrete divergence operator $\diver_h : \mathbb{R}^{2N_h} \to \mathbb{R}^{N_h}$ corresponding to \eqref{div_fd} is such that
\begin{equation}
\diver_h(p) = D_1 p_1 + D_2 p_2
= \begin{bmatrix} D_1 & D_2 \end{bmatrix} 
\begin{pmatrix} p_1 \\ p_2 \end{pmatrix}.
\end{equation}
We also denote by $\diag(V) \in \mathbb{R}^{N_h \times N_n}$ 
the diagonal matrix whose diagonal entries correspond to the components of $V$.

For $k \in \I^*_{\Delta t}$, we define the vectors 
$U^{n,k}, U^{N_t} \in \mathbb{R}^{{N_h^2}}$ by
\begin{equation}\label{eq:defU}
U^{n,k}_{i + j N_h} = u^{n,k}_{i,j}, \qquad 
U^{N_t}_{i + j N_h} = G(x_{i,j}) \quad \forall\, i,j \in \I_h.
\end{equation}
Similarly, for $k \in \I_{\Delta t} \setminus \{0\}$, we set 
$M^{n,k}, M^0 \in \mathbb{R}^{{N_h^2}}$ as
\begin{equation}\label{eq:defM}
M^{n,k}_{i + j N_h} = m^{n,k}_{i,j}, \qquad 
M^0_{i + j N_h} = m_0(x_{i,j}) \quad \forall\, i,j \in \I_h.
\end{equation}
Finally, for $k \in \I^*_{\Delta t}$, 
we define $Q^{n,k}=(Q^{1,n,k},Q^{2,n,k}) \in \RR^{2{N_h^2}}$ as
\begin{equation}\label{eq:defQ}
Q^{1,n,k}_{i + j N_h} = q^{1,n,k}_{i,j},\quad Q^{2,n,k}_{i + j N_h} = q^{2,n,k}_{i,j} \quad \text{for all } i,j \in \I_h.
\end{equation}

With this notation, scheme~\eqref{scheme-B} can be written in matrix form as
\begin{equation}
\label{SLmatrix_f}
\begin{cases}
U^{n,k} = \mathcal{A}(Q^{n,k}) \, U^{n,k+1} + \Delta t \, \mathcal{W}^k M^{n,k} + \Delta t \, \mathcal{B}^k, 
& \text{for all } k \in \I^*_{\Delta t}, \\[2mm]
U^{n,N_t} = U^{N_t},
\end{cases}
\end{equation}
where
\begin{align}
&\label{A(Q)}(\mathcal A(Q^{n,k}))_{iN_h+j,p+N_hq}=\frac{1}{2d}  \sum_{\ell=1}^4\beta_{p,q}(y^{\ell}(x_{i,j},q^{n,k}_{i,j})), \quad\text{for all }i,j,p,q\in\I_{h}\\[6pt]
&\label{Wk}
(\mathcal{W}^k)_{iN_h+j,p+N_hq} = (\operatorname{diag}\big(F'(M^{n-1,k})\big))_{iN_h+j,p+N_hq},
\\[6pt]
		&\label{Bk}(\mathcal B^k)_{i+N_hj}=\frac{|q^{n,k}_{i,j}|^2}{2}-V(x_{i,j})\nonumber\\
		&\hspace{2cm}+F((M^{n-1,k})_{i+N_hj})-F'((M^{n-1,k})_{i+N_hj})(M^{n-1,k})_{i+N_hj}\quad
	\text{for all }i,j\in\I_{h}.
\end{align}
Scheme \eqref{scheme-F} can be written in matrix form as 
 \begin{equation}
 \label{SLmatrix_b}
\left\{
\begin{array}{ll}
M^{n,k+1}=\mathcal A^*(Q^{n,k})M^{n,k}+\Delta t\mathcal Z^{k+1}U^{n,k+1}+\D t\mathcal C^{k+1}\quad
&k\in\I^*_{\D t},\\[6pt]
M^{n,0}=M^0,
\end{array}
\right.
\end{equation}
where $\mathcal A(Q)^*$ denotes the transpose of $\mathcal A(Q)$ given by \eqref{A(Q)}, for every $k\in\I^*_{\D t}$, $\mathcal Z^{k+1}$ is ${N_h^2}\times{N_h^2}$ matrix such that, given $V \in \RR^{{N_h^2}}$ and let $\mathscr{M}=M^{n-1,k+1}$
\begin{equation}
\label{Z_matrix}
\mathcal{Z}^{k+1} V = \diver_h \big( \diag(\mathscr{M}) \, D_h U \big) 
=\diag(\mathscr{M})\diver_h (D_hV)+ \diag(D_1 \mathscr{M}) D_1 V + \diag(D_2 \mathscr{M}) D_2 V.
\end{equation}
This  shows that $\mathcal{Z}^{k+1}$
can be written in  matrix form as
\begin{equation}
\label{Zk}
\mathcal{Z}^{k+1}= \diag(\mathscr{M}) \, L_h + \diag(D_1 \mathscr{M}) D_1 + \diag(D_2 \mathscr{M}) D_2.
\end{equation}
where
$L_h := D_1D_1 + D_2 D_2$ is the discrete 2D Laplacian operator.
Moreover $\mathcal C^{k+1}$ is the vector in $\RR^{{N_h^2}}$  
\begin{align}
\label{Ck}
(\mathcal C^{k+1})_{i+N_hj}=-(\mathcal{Z}^{k+1} U^{n-1,k+1})_{i+N_hj}.
\end{align}

\begin{remark}[Mass preservation]
The semi-Lagrangian adjoint operator $\mathcal A^*(Q^{n,k})$ is the transpose of a row-stochastic matrix $\mathcal A(Q^{n,k})$ representing linear interpolation along characteristics. Let $(i,j),(p,q)\in\mathcal I_h$ and define the linear indices
$
\zeta = i + N_h j,\,l = p + N_h q,
$
then let \(A\) be the  matrix with entries $  A_{\zeta,l} := \big(\mathcal A(Q^{n,k})\big)_{\zeta,l}.$ By definition,
\[
A_{\zeta,l}\ge 0, \qquad \sum_{l} A_{\zeta,l}= 1 \quad \forall \,\zeta,
\]
so $A$ is row-stochastic. Its adjoint satisfies $(A^*)_{l,\zeta} = A_{\zeta,l}$, and for any discrete density $M$ we have
\[
\sum_{l} (A^* M)_{l} = \sum_{l} \sum_{\zeta} A_{\zeta,l} M_{\zeta} 
= \sum_{\zeta} M_{\zeta} \sum_{l} A_{\zeta,l} 
= \sum_{\zeta} M_{\zeta}.
\]
Furthermore, the linearization term $\mathcal{Z}^{k+1}(U^{n,k+1}-U^{n-1,k+1})$, defined via the discrete divergence $\text{div}_h$ \eqref{div_fd}, sums to zero under periodic boundary conditions. 
Hence, if $\sum_{\zeta} M^{n,0}_{\zeta}=1$, then for all $k$,
\[
\sum_{\zeta} M^{n,k}_{\zeta} = 1,
\]
so the scheme is mass-preserving.
\end{remark}
\begin{remark}[Positivity]
The linearization term $\mathrm{div}\!\big(m^{n-1}(Du^n - Du^{n-1})\big)$ acts as a source and is generally not sign-definite. 
Although the transport part preserves positivity,
\[
M^{n,k}_l \ge 0,\;\forall\,l\Longrightarrow\; (\mathcal{A}(Q^{n,k})^* M^{n,k})_l \ge 0,\forall\,l
\]
the full scheme may produce small negative values of order $O(\Delta t)$ due to the source term, so global non-negativity of the density cannot be guaranteed.
\label{posi}
\end{remark}
\begin{remark}[Stability of the scheme \eqref{SLmatrix_b}]
If the source term in \eqref{SLmatrix_b}
satisfies 
\[
\big\| \mathcal{Z}^{k+1} (U^{n,k+1}-U^{n-1,k+1}) \big\|_{\ell^1} \le C,
\] 
for a positive $C>0$, we can show that the scheme \eqref{SLmatrix_b} is stable. Indeed, since adjoint operator $\mathcal A^*(Q^{n,k})$ preserves mass and has positive entries, we obtain
\[
\| M^{n,k+1} \|_{\ell^1} \le \| M^{n,k} \|_{\ell^1} + C \, \Delta t,
\] 
and by iterating backward in time, we get
\[
\| M^{n,k+1} \|_{\ell^1} \le \| M^{n,0} \|_{\ell^1} + C \, T.
\]
The uniform bound of the source term  is a crucial step, as it would allow one to prove convergence of the full discrete scheme. We leave this analysis for future work.\end{remark}
The final discrete Newton iteration system reads as follows. Given 
$(U^{n-1,k}, M^{n-1,k}, Q^{n,k})_{k \in \I_{\Delta t}}$, find 
$(U^{n,k}, M^{n,k})_{k \in \I_{\Delta t}}$ such that
\begin{equation}
\label{fully_newton}
    \begin{cases}
    U^{n,k}=\mathcal A(Q^{n,k})U^{n,k+1}+\Delta t\mathcal W^{k}M^{n,k}+\D t\mathcal B^{k}\quad
&k\in\I^*_{\D t},\\[4pt]
M^{n,k+1}=\mathcal A^*(Q^{n,k})M^{n,k}+\Delta t\mathcal Z^{k+1}U^{n,k+1}+\D t\mathcal C^{k+1}\quad
&k\in\I^*_{\D t},\\[4pt]
U^{n,N_t}=U^{N_t},\\[4pt]
M^{n,0}=M^0,
    \end{cases}
\end{equation}
where $\mathcal{A}(Q^{n,k}),\;\mathcal W^{k},\;\mathcal B^{k},\;\mathcal Z^{k+1}$, and $\mathcal C^{k+1}$ are given by \eqref{A(Q)}, \eqref{Wk}, \eqref{Bk}, \eqref{Zk}, and \eqref{Ck} respectively.

\subsection{Well-posedness}
To establish the well-posedness of the system \eqref{fully_newton}, we represent it in matrix form. 
For this purpose, let $\bar{U}$ and $\bar{M}$ be the vectors in $\RR^{(N_t+1)N_h^2}$ defined by
\begin{equation}
\label{notation}
(\bar U^n)_{k N_h^2 + i N_h + j} = (U^{n,k})_{i+N_h j},\quad
(\bar M^n)_{k N_h^2 + i N_h + j} = (M^{n,k})_{i+N_h j},\quad
k\in\I_{\Delta t}, \; i,j\in\I_h.
\end{equation}

Next, define the block matrices $\mathbf{A}$ and $\mathbf{W}$ as

\begin{equation*}
\mathbf{A} = 
\begin{bmatrix}
\mathcal{I} & -\mathcal{A}^0      & \mathcal{O}      & \cdots           & \mathcal{O} \\[4pt]
\mathcal{O} & \mathcal{I}         & -\mathcal{A}^1   & \ddots           & \vdots      \\[4pt]
\vdots      & \ddots              & \ddots           & \ddots           & \vdots      \\[4pt]
\vdots      & \ddots              & \ddots           & \mathcal{I}      & -\mathcal{A}^{N_t-1} \\[4pt]
\mathcal{O} & \cdots              & \cdots           & \mathcal{O}      & \mathcal{I}
\end{bmatrix},\qquad
\mathbf{W} = \Delta t\,
\begin{bmatrix}
\mathcal{W}^0 & \mathcal{O}        & \cdots & \cdots          & \mathcal{O} \\[4pt]
\mathcal{O}    & \mathcal{W}^1     & \ddots &                  & \vdots      \\[4pt]
\vdots         & \ddots             & \ddots & \ddots           & \vdots      \\[4pt]
\vdots         &                    & \ddots & \mathcal{W}^{N_t-1} & \mathcal{O} \\[4pt]
\mathcal{O}    & \cdots             & \cdots & \mathcal{O}     & \mathcal{O}
\end{bmatrix}.
\end{equation*}

Here, $\mathcal{I}$ and $\mathcal{O}$ denote the $N_h^2 \times N_h^2$ identity and zero matrices, respectively. 
If $F'>0$ and $m>0$, then the first $N_t$ diagonal blocks of $\mathbf{W}$ are positive, 
provided that $M^{n-1}$ is positive.
We also define the block diagonal matrix $\mathbf{Z}$ as
\begin{equation*}
\mathbf{Z} = \Delta t \, 
\begin{bmatrix}
\mathcal{O} & \cdots& \cdots & \cdots & \mathcal{O} \\[4pt]
\mathcal{O} & \mathcal{Z}^1 & \mathcal{O} & \cdots & \vdots\\[4pt]
\vdots & \ddots& \ddots & \ddots & \vdots \\[4pt]
\vdots & \ddots & \ddots & \ddots & \mathcal{O} \\[4pt]
\mathcal{O} & \cdots& \cdots & \mathcal{O} & \mathcal{Z}^{N_t}
\end{bmatrix},
\end{equation*}
where, for every $k \in \I^*_{\Delta t}$, $\mathcal{Z}^k$ are defined in \eqref{Zk}.

Finally, let
$$
\mathbf{B} = \Delta t \, \big[\mathcal{B}^0, \dots, \mathcal{B}^{N_t-1}, \frac{1}{\Delta t} U^{N_t}\big]^\top, 
\quad
\mathbf{C} = \Delta t \, \big[\frac{1}{\Delta t} M^0, \mathcal{C}^0, \dots, \mathcal{C}^{N_t-1}\big]^\top,
$$
with $\mathcal{B}^k$ and $\mathcal{C}^k$ given in \eqref{Bk} and \eqref{Ck}, respectively.

Then, the discrete Newton system \eqref{fully_newton} can be written as the following system:

\begin{equation}
\label{matrix-form}
\begin{bmatrix}
\mathbf{A} & -\mathbf{W} \\[2mm]
-\mathbf{Z} & \mathbf{A}^{\top}
\end{bmatrix}
\begin{bmatrix}
\bar{U}^n \\[1mm]
\bar{M}^n
\end{bmatrix}
=
\begin{bmatrix}
\mathbf{B} \\[1mm]
\mathbf{C}
\end{bmatrix}.
\end{equation}

\begin{proposition}
\label{propcons}
Assume that $F'>0$ and $\bar{M}^{n-1}>0$ for any $n\in\NN$. Then, there exists a unique solution $(\bar{U}^n,\bar{M}^n)$ to the system \eqref{matrix-form}.
\end{proposition}

\begin{proof}
Suppose that ${\textbf{B}}=0$ and ${\textbf{C}}=0$, then \eqref{matrix-form} reads as
\begin{equation}\label{null_system}
		\left\{	
\begin{array}{l}
\textbf{A}\,{\bar{U}^n} -\textbf{W}{\bar{M}^n}=0,  \\ -\textbf{Z}\,{\bar{U}^n}+\textbf{A}^{\top}{\bar{M}^n}=0
\end{array} 
\right.
\end{equation}
Multiplying the first equation by $(\bar M^n)^{\top}$ and the second one by $(\bar U^n)^{\top}$ one gets
\begin{equation*}	
\left\{ 	
\begin{array}{l}
(\bar{M}^n)^{\top} \textbf{A}\,{\bar{U}^n} -(\bar{M}^n)^{\top} \textbf{W}{\bar{M}^n}=0\\
 -(\bar{U}^n)^{\top} \textbf{Z}\,{\bar{U}^n}+(\bar{U}^n)^{\top} \textbf{A}^{\top}{\bar{M}^n}=0.		
	\end{array} 
	\right.
\end{equation*}
{Subtracting} both equations, we obtain
\begin{equation}\label{matrix_identity}
(\bar{M}^n)^{\top} \textbf{W}{\bar{M}^n}- (\bar{U}^n)^{\top} \textbf{Z}\,{\bar{U}^n}=0.
\end{equation}
Since $F'>0$, the block $\mathcal{W}^k$ is positive definite for all $k\in\I_{\D t}$. Moreover,  the block $\mathcal Z^k$ is negative definite for all $k\in\I_{\D t}$, since  $\bar{M}^{n-1}>0$ and it is the sum of a negative definite matrix and a skew symmetric matrix  and then, by \cite[Remark 1]{J}, it is negative definite.

Hence, it follows from \eqref{matrix_identity} that
\be
\label{bar u}
(\bar{U}^n)_{kN_h^2+iN_h+j}=0\quad\text{for all }k\in\I^*_{\D t},\, i,j\in\I_h,
\ee
and
\be
\label{bar m}
(\bar{M}^n)_{kN_h^2+iN_h+j}=0\quad\text{for all }k\in\I^*_{\D t}\, i,j\in\I_h.
\ee
Finally, replacing \eqref{bar u} and \eqref{bar m} in \eqref{null_system}, we get $\bar{U}^n= 0$ and $\bar{M}^n=0$, and then the existence of a unique solution to \eqref{matrix-form}.
\end{proof}
\begin{remark}
The assumption $F'>0$ is a standard monotonicity condition in MFG with separable Hamiltonian, which guarantees uniqueness of the solution to~\eqref{MFGg}; see~\cite[Remark~2.1]{newton}. 
In the second-order case ($\nu > 0$), the diffusion ensures that the density remains positive. At the discrete level, the condition $M^{n-1} > 0$ should be interpreted as a positivity requirement on the density at the previous Newton iterate. However, the numerical scheme does not, in general, verify this property (see Remark~\ref{posi}) and in practice this condition may fail particularly when $\nu$ is small. In Section~\ref{numerics}, we propose a framework to address this issue.
\end{remark}

\begin{remark}
Note that the blocks of $\textbf W$ would be dense matrices if $F(m(t,x))$ is replaced by a nonlocal operator $f[m(t,\cdot)]$. 
In this case, we use the notation $\frac{\delta f}{\delta m}$
for the flat derivative of $f$ (see e.g \cite{mfgnash}).
The condition $F'>0$ can then be replaced by assuming that $\frac{\delta f}{\delta m}$ is Lipschitz continuous and that $f$ satisfies
 $$
\begin{aligned}
\int_{\mathbb T^d} (f[m](x) - f[m'](x)) \, \dd(m - m')(x) &\geq 0\quad \text{for all }m,m'.
\end{aligned}
$$
A typical example is a nonlocal coupling with smoothing effect 
\[
f(x,m) = \int_{\mathbb{T}^d} \Phi(z,(\rho\ast m)(z))\rho(x-z)\dd z,
\]
where $\ast$ denotes the usual convolution on $\mathbb{T}^d$,  
$\Phi$ is a smooth and nondecreasing with respect to $m$, and $\rho$ is a smooth, even function with compact support. In this case, writing $\Phi=\Phi(x,\theta)$,  the flat derivative of $f$ reads
$$
\frac{\delta f}{ \delta m}(x, m, y) = \int_{\mathbb{T}^d} \frac{\partial\Phi}{\partial \theta}(z,(\rho\ast m)(z))\rho(x-z)\rho(z-y)\dd z.
$$
\end{remark}

\section{A finite differences scheme}
\label{upwind}

In this section, we discretize the iterative system \eqref{eq:Newton} using an implicit finite difference scheme in two dimensions and prove its well-posedness. 
The equation for $v$ is discretized  using upwind differences for the drift term and central differences for the second-order spatial derivatives. 
The equation for $m$ is then approximated by taking the adjoint of the resulting scheme.

In order to discretize the system, we define the discrete time derivative of $v\in B(\mathcal G_{\Dt}\times\G_h)$  as
\be
D_{t}v^k_{i,j}=\frac{v^{k+1}_{i,j}-v^k_{i,j}}{\Dt}\quad\text{for all }k\in\I_{\D t}^*,\;i,j\in\I_{h},
\ee
Given $q^n
\in B(\G_{\D t}\times\G_h)^2$ and $u^{n-1}, m^{n-1}, m^{n-1}\in B(\G_{\D t}\times\G_h)$, 
we discretize the first equation in \eqref{eq:Newton} using an implicit scheme finite difference scheme, which reads for all $k\in\I_{\D t}^*$ and $i,j\in\I_h$ as
\begin{equation}
\label{upwind1}
\begin{cases}
\ds -D_{t}u^{n,k}_{i,j}-\nu (L_hu^{n,k})_{i,j}+q^{n,k}_{i,j}(D_hu^{n,k})_{i,j}=&\frac{1}{2}|(q^{n,k})_{i,j}|^2+V(x_{i,j})
+\mathcal{F}_{i,j}(m^{n,k+1},m^{n-1,k+1}),\\[6pt]
u^{n,T}_{i,j}=G(x_{i,j}),
\end{cases}
\end{equation}
 where
$L_h$ is thediscrete Laplace operator, already defined in Section \ref{sec:Newton}, and
for every  $\zeta,\tilde \zeta\in B(\G_h)$ 
$$
\mathcal{F}_{i,j}(\zeta,\tilde\zeta)=F'(\tilde \zeta_{i,j})(\zeta_{i,j}-\tilde \zeta_{i,j})+F(\tilde \zeta_{i,j}).
$$
Then, we consider the following approximation of the forward equation in \eqref{eq:Newton}
\begin{equation}
\label{upwind2}
\begin{cases}
    D_{t}m^{n,k}_{i,j}-\nu(L_hm^{n,k+1})_{i,j}-(\diver_h(q^{n,k}m^{n,k+1}))_{i,j}=\diver_h(m^{n-1,k+1}(D_hu^{n,k}-D_hu^{n-1,k}))_{i,j},\\[6pt]
m^{n,0}_{i,j}=m_0(x_{i,j}).
\end{cases}
\end{equation}
\begin{remark}
Notice that, qiven $q
\in B(\G_{\D t}\times\G_h)^2$, the operator $\mathcal B(\G_h)\ni m\to(-\nu(L_h m)-(\diver_h(qm))\in\mathcal B(\G_h)$ is the adjoint of the operator $\mathcal B(\G_h)\ni u\to(-\nu(L_hu)+q\,D_h u)\in \mathcal B(\G_h)$.
\end{remark}

Combining \eqref{upwind1} and \eqref{upwind2}, and using the vector notations \eqref{eq:defU}-\eqref{eq:defQ}, we get the following fully discrete upwind scheme for the Newton iterations system \eqref{eq:Newton}
\begin{equation}
\label{form_matrix}
    \begin{cases}
   \mathcal D^{k}U^{n,k}=U^{n,k+1}+\D t\mathcal W^{k+1}M^{n,k+1}+\D t\tilde{\mathcal B}^{k+1}\quad &\text{for }k\in\I^*_{\D t},\\[4pt]
    (\mathcal D^{k})^{\top}M^{n,k+1}=M^{n,k}+\D t\mathcal Z^{k}U^{n,k}+\D t\tilde{\mathcal C}^{k+1}\quad &\text{for }k\in\I^*_{\D t},\\[4pt]
    U^{n,N_t}=U^T,\\[4pt]
    M^{n,0}=M^0,
    \end{cases}
\end{equation}
where $\mathcal D^k$ is the ${N_h^2}\times{N_h^2}$ matrix
$$\mathcal D^k=\mathcal{I}-\D t(\nu L_h + Q_1D_1+Q_2D_2),$$
and $\tilde{\mathcal C}^{k+1}, \tilde{\mathcal B}^{k+1}$ are vectors in  $\mathbb{R}^{N_h^2}$  defined as
\begin{equation*}
\begin{split}
(\tilde {\mathcal C}^{k+1})_{i+N_hj}&=-(\mathcal{Z}^{k+1} U^{n-1,k})_{i+N_hj},\\
    (\tilde{\mathcal B}^{k+1})_{i+N_hj}&=\frac{|(q^{n,k})_{i,j}|^2}{2}-V(x_{i,j})\\
		&\;+F((M^{n-1,k+1})_{i+N_hj})-F'((M^{n-1,k+1})_{i+N_hj})(M^{n-1,k+1})_{i+N_hj}\quad
	\text{for }i,j\in\I_{h}.
\end{split}
\end{equation*}

Finally, define $\tilde{U}$ and $\tilde{M}\in\RR^{(N_t)N_h^2}$ by
\begin{equation*}
\begin{split}
(\tilde U^n)_{kN_h^2+iN_h+j}&=( U^{n,k})_{i+N_hj}\quad\text{for }k\in\I^*_{\D t},\, i,j\in\I_h,\\
(\tilde M^n)_{kN_h^2+iN_h+j}&=( M^{n,k})_{i+N_hj}\quad\text{for }k\in\I_{\D t}\setminus\{0\},\, i,j\in\I_h.
\end{split}
\end{equation*}

In order to write system \eqref{form_matrix} in a matrix way, as in \eqref{matrix-form}, we define first the matrices
$$
\tilde{\textbf A}=
\left[ \begin{matrix}
	\mathcal D^0 & -\mathcal{I} &\mathcal{O} &\cdots & \mathcal{O} \\
		\mathcal{O}& 	\mathcal D^1 & \ddots & \ddots& \vdots  \\
		\vdots & \ddots & \ddots & \ddots&\vdots  \\
		\vdots  &\ddots  &\ddots  &\ddots &- \mathcal{I} \\
		\mathcal{O}  &\cdots  & \cdots & \mathcal{O} &	\mathcal D^{N_t-1}
	\end{matrix} \right],\quad
\tilde{\textbf{W}}=\D t	\left[\begin{matrix}
\mathcal W^1 & \mathcal{O} & \cdots &\cdots & \mathcal{O} \\
\vdots & \mathcal W^2 & \ddots &\ddots & \mathcal{O} \\
\mathcal{O} & \ddots &\ddots& \ddots & \vdots \\
\vdots & \ddots & \ddots &\ddots&  \vdots \\
\mathcal{O} & \cdots & \cdots & \cdots & \mathcal W^{N_t}\\
\end{matrix}\right],
$$
and
$$
\tilde{\textbf{Z}}=\D t	\left[\begin{matrix}
\mathcal Z^0 & \mathcal{O} & \cdots &\cdots & \mathcal{O} \\
\mathcal{O} & \mathcal Z^1 & \ddots &\ddots & \mathcal{O} \\
\vdots & \ddots &\ddots& \ddots & \vdots \\
\vdots & \ddots & \ddots &\ddots&  \vdots \\
\mathcal{O} & \mathcal{O} & \cdots & \cdots & \mathcal Z^{N_t-1}\\
\end{matrix}\right].
$$
Hence, \eqref{form_matrix} is equivalent to the  system
{\begin{equation}
\label{matrix_form_fd}
	\left[
	\begin{matrix}
		\tilde{\textbf A}  &- {\tilde {\textbf{W}}}\\[6pt]
		- {\tilde {\textbf{Z}}} &\tilde{\textbf A}^*
	\end{matrix}\right]\left[
	\begin{matrix}
		\tilde{U}^{n}\\[6pt]
		\tilde{M}^n
	\end{matrix}\right]=
	\left[
	\begin{matrix}
	\tilde{\textbf{B}}\\[6pt]
	\tilde{\textbf{C}}
	\end{matrix}\right],
\end{equation}}
where $\tilde{\textbf{B}}=\D t\Big[\mathcal B^1,\dots,\mathcal B^{N_t}+\frac{1}{\D t}U^{N_t}\Big]^{\top}$ and $\tilde{\textbf{C}}=\D t\Big[\frac{1}{\D t}M^0+\mathcal C^1,\dots, \mathcal C^{N_t}\Big]^{\top}$. Then, arguing as in the proof of Proposition \ref{propcons}, one can show the following well-posedness result.
\begin{proposition}
Suppose that $\tilde M^{n-1}>0$ for any $n\in\NN$. Then, there exists a unique solution $(\tilde U^n,\tilde M^n)$ to the system \eqref{matrix_form_fd}.
\end{proposition}


\section{Numerical tests}
\label{numerics}
In this section, we assess the performance of the proposed numerical schemes  through a series of tests in both one and two spatial dimensions, 
considering cases with separable and non-separable Hamiltonians. 
In all tests, the Newton iterations are terminated when the following error metrics 
fall below a prescribed tolerance~$\tau$, set to $10^{-4}$ in all the experiments.
Specifically, we monitor the relative variation between two consecutive Newton iterates 
for both the value function and the distribution, according to:
\begin{equation}
\label{err}
E(m^n):=\|m^{n+1}-m^n\|_\infty \quad \text{and} \quad E(u^n):=\|u^{n+1}-u^n\|_\infty.
\end{equation}
The systems defined by equations \eqref{matrix-form} and \eqref{matrix_form_fd} 
are solved using Block Gauss--Seidel iterations. 
The iterations are terminated once the uniform norms of the differences between consecutive solutions 
fall below a prescribed threshold $\delta$, which is set to $10^{-4}$. 
This ensures comparable accuracy for both discretization procedures.
We compare three approaches for solving system \eqref{MFGg}: the semi-Lagrangian scheme \eqref{fully_newton} ({Newton--SL}), 
the finite difference scheme \eqref{form_matrix} ({Newton--FD}), 
and the finite difference schemes from \cite{finitedifference1} solved via Newton iterations ({FD--Newton}). 
Algorithm~\ref{algo_FD} implements the Newton method for both Newton--SL and Newton--FD.
We note that the semi-Lagrangian scheme used in this work is explicit and does not require a parabolic CFL condition, 
whereas the finite difference schemes are implicit and thus impose no restriction on the time step. 
For accuracy purposes, the semi-Lagrangian scheme is run with a time step of the form $\Delta t = O(h^{3/2})$ 
(see \cite{Ferretti1,Ferretti2} for a detailed analysis).
We now present a series of four numerical tests designed to illustrate the behavior and performance of the proposed Newton-based schemes. 
The first test considers a one-dimensional MFG with a known reference solution, allowing us to validate the accuracy of the schemes and compare the performance of Newton--SL, Newton--FD, and FD--Newton. 
The second test focuses on one-dimensional MFG with smooth initial and terminal conditions, considered for two different diffusion coefficients ($\nu = 0.4$ and $\nu = 0.02$), highlighting the stability and robustness of the schemes. 
In the small-diffusion case ($\nu = 0.02$), we observe a break-down of the standard Newton iterations for the finite-difference schemes (Newton--FD and FD--Newton) due to a poor initialization, whereas Newton--SL remains stable. 
To overcome this difficulty, we employ a globalized version of Newton's method, which is activated whenever the local iteration fails; this global strategy is described in detail in the following subsection. 
Given the robust and accurate performance of the Newton--SL scheme observed in the tests 1--2, 
the subsequent experiments involving non-separable Hamiltonian
are conducted using only this discretization approach.
The third test examine one-dimensional MFG system with non-separable Hamiltonian, and finally the fourth test examine a two-dimensional MFG systems with a separable Hamiltonian. 
Both tests demonstrate the flexibility and robustness of the Newton--SL scheme in handling more general and nonlinear Hamiltonian structures.
Together, these tests provide a comprehensive evaluation of the schemes, highlighting their accuracy, efficiency, and stability across different problem types, dimensions, and diffusion regimes.

\begin{algorithm}[H]
\caption{Local Newton iterations for mean field games}
\label{algo_FD}
\begin{algorithmic}[1]
\State \textbf{Input}: Initial guesses $ u^0$, $ m^0$, and tolerance $\tau$ 
    \State $q^0\gets D_hu^0$ 

\State $n\gets 0$
   \Repeat
         \State Compute $(m^{n+1}, u^{n+1})$ as the solution of~\eqref{fully_newton} (Newton--SL) or~\eqref{form_matrix} (Newton--FD) 

        \State $E(m^n) \gets \| m^{n+1} -  m^{n}\|_\infty $
        \State $E(u^n) \gets \| u^{n+1} -  u^{n}\|_\infty $
        \State Update $q^n\gets  D_hu^n$  
        \State $n \gets n + 1$
      \Until{ $E(m^n) \geq  \tau$ \textbf{or} $E(u^n) \geq \tau$}

    \State \textbf{Output}: Approximated solution $(u^{n+1},m^{n+1})$ to the MFG system \eqref{MFGg}
\end{algorithmic}
\end{algorithm}

\begin{remark}[Block Gauss--Seidel for the discrete Newton system\eqref{matrix-form}  ]

To construct a Block Gauss--Seidel iteration, we perform a splitting of the Newton matrix \eqref{matrix-form} as
\begin{equation}
\label{eq:block_splitting_correct}
M = \begin{bmatrix} \mathbf A & \mathbf0 \\ -\mathbf{Z} & \mathbf A^\top \end{bmatrix}, \qquad
N = \begin{bmatrix} \mathbf0 & \mathbf W \\ \mathbf 0 & \mathbf0 \end{bmatrix}.
\end{equation}
The corresponding Block Gauss--Seidel iteration is, for $\ell\geq0$,
\begin{equation}
\label{eq:block_GS_correct}
X^{(\ell+1)} = M^{-1} N X^{(\ell)} + M^{-1} b, \qquad 
X ^{(\ell)}= \begin{bmatrix} \bar U^{(\ell)} \\ \bar M ^{(\ell)}\end{bmatrix}, \quad b = \begin{bmatrix} \mathbf B \\ \mathbf C \end{bmatrix}.
\end{equation}
given an initial condition $X^{(0)}$.
\medskip

\noindent
Each diagonal block of $M$ corresponds to the natural temporal propagation of $U$ backward in time and $M$ forward in time, while the coupling terms $\mathbf W$ and $\mathbf Z$ are contained in $N$ and in the bottom-left block of $M$ respectively. 
The iteration can be implemented as a \emph{backward-forward sweep}:

\begin{enumerate}
    \item \emph{Backward sweep for $U$:}
    \[
    U^{n,N_t}_{(\ell+1)} = U^{N_t}, \quad 
    U^{n,k}_{(\ell+1)} = \mathcal A(Q^{n,k}) U^{n,k+1}_{(\ell+1)} + \Delta t \, \mathcal W^k M^{n,k}_{(\ell)} + \Delta t \, \mathcal B^k, 
    \quad k = N_t-1,\dots,0.
    \]
    \item \emph{Forward sweep for $M$:}
    \[
    M^{n,0}_{(\ell+1)} = M^0, \quad 
    M^{n,k+1}_{(\ell+1)} = \mathcal A^*(Q^{n,k}) M^{n,k}_{(\ell+1)} + \Delta t \, \mathcal Z^{k+1} U^{n,k+1}_{(\ell+1)} + \Delta t \, \mathcal C^{k+1}, 
    \quad k = 0,\dots,N_t-1.
    \]
\end{enumerate}

\medskip
\noindent 
The block-diagonal matrix $M$ is invertible because $\mathbf A$ and $\mathbf A^\top$ are invertible. The convergence of the Block Gauss--Seidel is guaranteed if 
$\rho( (\mathbf A^\top)^{-1}\mathbf Z \mathbf A^{-1}\mathbf W)<1$.
It can be shown that $\|\mathbf A^{-1}\|_{\infty}\leq CN_t$,
$\|\mathbf W\|_{\infty}\leq C\Delta t$ and that 
$
\|\mathbf{Z}\|_\infty 
= O\!\left(\frac{\Delta t}{h^2}\right),
$
however these bounds are not sufficient condition to ensure convergence, and it seems difficult to improve these estimates in general. 
Nonetheless, in all our numerical experiments, the Block Gauss--Seidel method converges systematically when the discrete density preserves positivity.


A comparison with GMRES and BiCGSTAB is reported in Table~\ref{tab:solver_comparison} of Section~\ref{sec:test2_1}. 
The results show that all the methods maintain convergence, however, Block Gauss--Seidel needs the lowest CPU time and the least number of iterations.
This behavior indicates that the effective spectral radius of the iteration matrix remains strictly smaller than one in the tested regime. 
\end{remark}


\noindent

\begin{remark}
For the {semi-Lagrangian scheme}, we approximate the drift term \( q^n \) using the following discrete formulation:
\[
q^{n,k}_{i,j} = \left( D\rho_\varepsilon * u^{n-1,k} \right)_{i,j}\quad \text{for all } k \in \mathcal{I}_{\Delta t},\, i,j \in \mathcal{I}_h,
\]
where \( * \) denotes the discrete circular convolution between two grid functions \( A, B \in \mathbb{R}^{N_h \times N_h} \) (see e.g \cite[Chapter 8]{oppenheim2010discrete}), and, for $\eps>0,$ the kernel \( \rho_\varepsilon \) is a Gaussian mollifier defined by
\[
\rho_\varepsilon(x) = \frac{1}{\sqrt{2\pi}\varepsilon} e^{-\frac{|x|^2}{2\varepsilon^2}},\quad \text{for all }x\in\RR^d.
\]
This approach, which is common in the framework of SL schemes (e.g see \cite{CSorder2,carlini2023lagrangegalerkin}),  enhances the accuracy of the scheme and allows us to avoid the constraints of a parabolic CFL condition, which would otherwise limit the time step.
For the {finite difference scheme}, we use instead the direct formulation:
\[
q^{n,k}_{i,j}  = (D_h u^{n-1,k})_{i,j} \quad \text{for all } k \in \mathcal{I}_{\Delta t},\, i,j \in \mathcal{I}_h,
\]
where \( D_h \) is given by \eqref{gradient_fd}.
\end{remark}

\subsection{Global Newton algorithm}\label{sec:GN}
In this section, we describe the global Newton method, which extends the local convergence properties of the standard Newton method to a global setting.  The key idea behind this approach   is to modify the Newton step when the current iterate is far from the solution, ensuring that each iteration makes sufficient progress. The global convergence of the Newton method in infinite dimensions is proved in \cite{global-newton}.
We adapt this framework to our MFG problem to enhance robustness and guarantee convergence even when starting from rough or distant initial guesses.
First, we define the \textit{search direction} $d^n:=(v^n,\rho^n)$ as the solution to the linear system 
\begin{equation}
\label{eq:Newton_it}
	J\cF(u^{n},m^{n})d^n=- \cF(u^{n},m^{n}).
\end{equation}
Then, the globalized Newton method generates sequences $\{(u^n, m^n)\}$, $\{d^n\}$, and $\{\alpha^n\}$, related by the iteration:
$$
(u^{n+1}, m^{n+1}) = (u^{n}, m^{n}) + \alpha^n d^n, \quad n\in\mathbb{N},
$$
where the step length $\alpha^n > 0$ is determined by a \textit{line-search} procedure using a suitably defined merit function. 
Following \cite{global-newton}, we use the squared $L^2$-norm of $\mathcal{F}$ as the merit function:
\be
\label{merit_f}
\ba{rcl}
\ds
\Theta(u,m):=\frac{1}{2}\|\mathcal{F}(u,m)\|_2^2&=&\int_{[0,T]\times \T^d}|-\partial_tu-\nu\Delta u+H(x,Du)-F(m(t,x)) |^2\dd x\dd t\\[10pt]
&+&\int_{[0,T]\times \T^d}|\partial_tm -\nu\Delta  m-\diver(mD_pH(x,Du)|^2\dd x\dd t\\[10pt]
&+&\int_{\mathbb T^d}|u(T,x)-G(x)|^2\dd x+\int_{\mathbb T^d}|m(0,x)-m_0(x)|^2\dd x.
\ea
\ee
A reasonable solution for the system should correspond to a root of the merit function, i.e., $\Theta(u, m) = 0$, which implies that $\mathcal{F}(u, m) = 0$ is satisfied, meaning that the system of equations is fulfilled.
The step length $\alpha^n$ is determined by enforcing the \textit{Armijo condition}, which guarantees a sufficient decrease of the merit function at each iteration. Specifically, this condition requires
$$
\Theta(u^{n+1}, m^{n+1}) \leq \Theta(u^n, m^n) + c \alpha^n \, \Theta'(u^n, m^n)(d^n),
$$
where  $0 < c < 1/2$  is a constant controlling the sufficient decrease, and $\Theta'$ denotes the Fréchet derivative of $\Theta$.  
Moreover, when the direction $d^n$ satisfies the Newton equation~\eqref{eq:Newton_it}, it holds that
\begin{equation}
\label{merit}
    \Theta'(u^n, m^n)d^n = -2\,\Theta(u^n, m^n) = -\|\mathcal{F}(u^n, m^n)\|^2.
\end{equation}
The global Newton procedure for solving \eqref{MFGg} is summarized in Algorithm~\ref{globalnewton}. Notice that the Newton step \eqref{eq:Newton_it} leads to the linear parabolic system for $d^n=(v^n,\rho^n)$, as given in \eqref{eq:Newton}. At the discrete level, the search direction $d^n$ can be computed using either the Newton-SL or the Newton-FD scheme. 

\begin{algorithm}[H]
\caption{Global Newton iterations for mean field games}
\label{globalnewton}
\begin{algorithmic}[1]
\State \textbf{Input}: Initial guesses $ u^0$, $ m^0$, parameters $\beta \in (0, 1)$, $c\in (0, \frac{1}{2})$, and tolerance $\tau>0$ 
    \State $q^0\gets D_hu^0$ 
\State $n\gets 0$
   \Repeat
   \State Compute $d^{n}= (m^{n+1/2}, u^{n+1/2})$ s the solution of~\eqref{fully_newton} (Newton--SL) or~\eqref{form_matrix} (Newton--FD)  
        \State Find the smallest integer $i_n \geq 0$ such that the step length $\alpha_n = \beta^{i_n}$ satisfies the Armijo condition:
            \begin{equation*}
            \Theta((m^{n}, u^{n}) + \alpha_n d^n) \leq \Theta(m^{n}, u^{n}) + c \alpha^n \Theta'(m^{n}, u^{n})d^n
           \end{equation*}
    \State Update the iteration
            $ (m^{n+1}, u^{n+1})\gets   (m^{n}, u^{n}) + \alpha_n d^n $
        \State $E(m^n) \gets \| m^{n+1} -  m^{n}\|_\infty $
        \State $E(u^n) \gets \| u^{n+1} -  u^{n}\|_\infty $
        \State Update $q^n\gets  D_hu^n$  
        \State $n \gets n + 1$
      \Until{ $E(m^n) \geq  \tau$ \textbf{or} $E(u^n) \geq \tau$}
    \State \textbf{Output}: Approximated solution $(u^{n+1},m^{n+1})$ to the MFG system \eqref{MFGg}
\end{algorithmic}
\end{algorithm}

\begin{remark}
The globalized Newton method typically requires more iterations than the local Newton method. Therefore, in our numerical experiments, we first apply the local Newton iterations, and only if they fail to converge, we switch to the globalized version.
\end{remark}

\subsection{Test 1: Accuracy Comparison for an MFG Problem with Reference Solution}
In this test (see \cite{popov}), we consider the MFG system \eqref{MFGg} on the time-space domain $[0,0.05]\times]0,1[$ with periodic boundary conditions, $\nu=0.05$, and $H(x,p)=\frac{|p|^2}{2}$. The coupling term is giving by 
\begin{equation*}
F(m(x))=4\min(4,m)-3m_0(x),\quad \text{for all } x\in ]0,1[,
\end{equation*}
while the initial condition is given by
\begin{equation*}
m_0(x)=
\begin{cases}
4\sin^2(2\pi(x-{1}/4))\quad &\text{if } x\in[{1}/4,{3}/4],\\[6pt]
0 &\text{otherwise},
\end{cases}    
\end{equation*}
and the terminal condition $G$ is set to  to zero.
To evaluate the accuracy errors, we compare the approximate solutions with a reference solution computed using a high-order value iteration scheme~\cite{Calzola_Carlini_Silva_lg_second_order} on a fine mesh with spatial step $h = 6.67 \times 10^{-4}$ and time step $\Delta t = h^{3/2}$. 
The accuracy of each scheme is assessed by computing the error in the discrete uniform norm. 
In our simulations, we set $\Delta t = h^{3/2} / 2$ for the Newton--SL method, and $\Delta t = h / 4$ for both Newton--FD and FD--Newton. It is worth noting that the semi-Lagrangian scheme is unconditionally stable, and the chosen time step is primarily dictated by accuracy considerations; in particular, we adopt the same time step used in~\cite{Calzola_Carlini_Silva_lg_second_order}, from which the reference solution is taken. 
On the other hand, the finite-difference scheme is implicit, allowing for relatively large time steps without any CFL restriction, so that the time step selection in this case is also driven by accuracy rather than stability constraints.

Table~\ref{table} reports the numerical results obtained for different grid resolutions $h$, including the errors, CPU times, and number of Newton iterations for the three methods under consideration. 
The discrete uniform errors are defined as
$$E_{\infty}(u) = \|u^n - u_{\mathrm{ref}}\|_{\infty}, \qquad 
E_{\infty}(m) = \|m^n - m_{\mathrm{ref}}\|_{\infty},
$$
where $(u_{\mathrm{ref}}, m_{\mathrm{ref}})$ denotes the reference solution, and $(u^n, m^n)$ denotes the numerical solution obtained by the Newton method at the iteration when the stopping criterion is satisfied.

We first observe that the Newton--SL approach requires fewer iterations and consistently less computational time compared to the two finite-difference-based schemes, while achieving an accuracy comparable to theirs. 
Between FD--Newton and Newton--FD, the results are very close in terms of both accuracy and CPU time, with Newton--FD showing a slight advantage in efficiency.

\begin{table}[htb!]
\centering
\small
\begin{tabular}{ |p{2cm}||p{2cm}||p{2cm}||p{1.4cm}||p{1.8cm}|}
 \hline
 \multicolumn{5}{|c|}{Newton-SL with $\D t=h^{3/2}/2$} \\[4pt]
 \hline
 $h$ & $E_\infty(u)$ & $E_\infty(m)$ & Time &  Iterations\\[4pt]
 \hline
2.50 $\cdot 10^{-2}$   & 5.51 $\cdot 10^{-2}$  & 1.64 $\cdot 10^{-1}$  & 0.61s & 6\\[4pt]
1.25 $\cdot 10^{-2}$  & 2.40 $\cdot 10^{-2}$   & 1.16 $\cdot 10^{-1}$ & 2.77s & 7\\[4pt]
6.25 $\cdot 10^{-3}$ & 1.83 $\cdot 10^{-2}$  & 6.61 $\cdot 10^{-2}$ & 13.92s & 7\\[4pt]
3.125 $\cdot 10^{-3}$   & 4.50 $\cdot 10^{-3}$  & 1.41 $\cdot 10^{-2}$ &80.60s & 7\\[4pt]
\hline
 \multicolumn{5}{|c|}{FD-Newton with $\D t=h/4$} \\[4pt]
 \hline
 $h$ & $E_\infty(u)$ & $E_\infty(m)$ & Time & Iterations\\[4pt]
 \hline
2.50 $\cdot 10^{-2}$   & 1.23 $\cdot 10^{-1}$  & 3.11 $\cdot 10^{-2}$  & 2.23s & 7\\[4pt]
1.25 $\cdot 10^{-2}$  & 6.21 $\cdot 10^{-2}$   & 1.63 $\cdot 10^{-2}$ & 18.32s & 8\\[4pt]
6.25 $\cdot 10^{-3}$ & 3.14 $\cdot 10^{-2}$  & 8.75 $\cdot 10^{-3}$ &  92.91s & 8\\[4pt]
3.125 $\cdot 10^{-3}$   & 1.77 $\cdot 10^{-2}$  & 9.54 $\cdot 10^{-3}$ & 597.21s & 8\\[4pt]
 \hline
 \multicolumn{5}{|c|}{Newton-FD with $\D t=h/4$} \\[4pt]
 \hline
 $h$ & $E_\infty(u)$ & $E_\infty(m)$ & Time &  Iterations\\[4pt]
 \hline
2.50 $\cdot 10^{-2}$   & 1.532 $\cdot 10^{-1}$  & 3.42 $\cdot 10^{-2}$  & 1.48s & 7\\[4pt]
1.25 $\cdot 10^{-2}$  & 6.71 $\cdot 10^{-2}$   & 1.83 $\cdot 10^{-2}$ & 12.27s & 7\\[4pt]
6.25 $\cdot 10^{-3}$ & 3.37 $\cdot 10^{-2}$  & 9.51 $\cdot 10^{-3}$ & 68.10s & 7\\[4pt]
3.125 $\cdot 10^{-3}$   & 1.91 $\cdot 10^{-2}$  & 7.38 $\cdot 10^{-3}$ & 436.01s & 7\\[4pt]
 \hline
 \end{tabular}
\vspace{0.2cm}
\caption{Test~1. Comparison of Newton--SL, FD--Newton, and Newton--FD schemes: errors, CPU times, and number of iterations for the solution $(u, m)$.}
\label{table}
\end{table}

\subsection{Test 2: a MFG problem varying the diffusion}
In this test, we consider two MFG systems, also numerically studied in~\cite{litest}, defined on the time--space domain $[0, 0.01] \times (0,1)$ with periodic boundary conditions and diffusion coefficients $\nu = 0.4$ and $\nu = 0.02$, respectively.
The Hamiltonian is given by $H(x,p) = \tfrac{|p|^2}{2}-V(x)$, and the corresponding data, illustrated in Figure~\ref{data}, are
\begin{equation*}
\ba{rcl}
m_0(x)&=& 1+\frac{1}{2} \cos(2\pi x),\\
G(x)&=&\sin(4\pi x)+0.1\cos(10\pi x),\\
V(x)&=&200\cos(2\pi x)-10\cos(4\pi x),\\
F(m)&=&m^2.
\ea    
\end{equation*}
The spatial and temporal discretization parameters are set to $h = 6.25 \cdot 10^{-3}$ and $\Delta t = h/4$ for the two finite-difference schemes, while $\Delta t = h^{3/2}/2$ is used for the Newton--SL method.

The purpose of this test is to assess the performance of the proposed methods in both a highly diffusive regime and a nearly deterministic setting, and to investigate how the diffusion parameter influences the convergence behavior and stability of the Newton iterations.
Furthermore, we provide a comparative analysis against standard iterative solvers to evaluate the relative computational efficiency and robustness of our approach across these different regimes.

\begin{figure}[htb!]
    \centering
    \begin{subfigure}[b]{0.32\textwidth}
        \includegraphics[width=\textwidth]{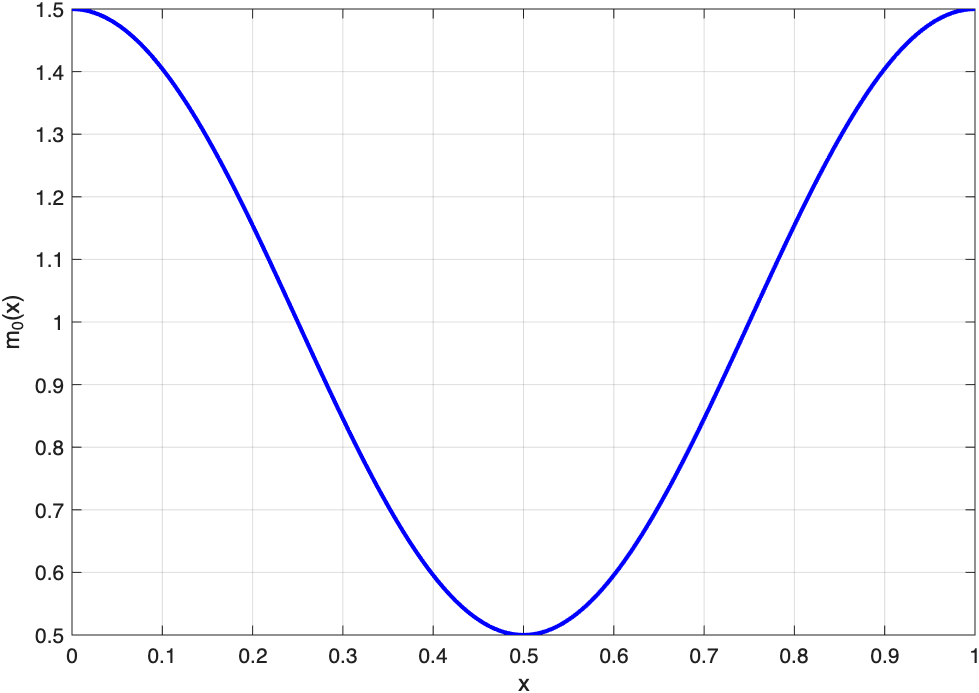}
        \caption{Initial distribution $m_0$}
        \label{fig:surb1}
    \end{subfigure}
    \begin{subfigure}[b]{0.32\textwidth}
        \includegraphics[width=\textwidth]{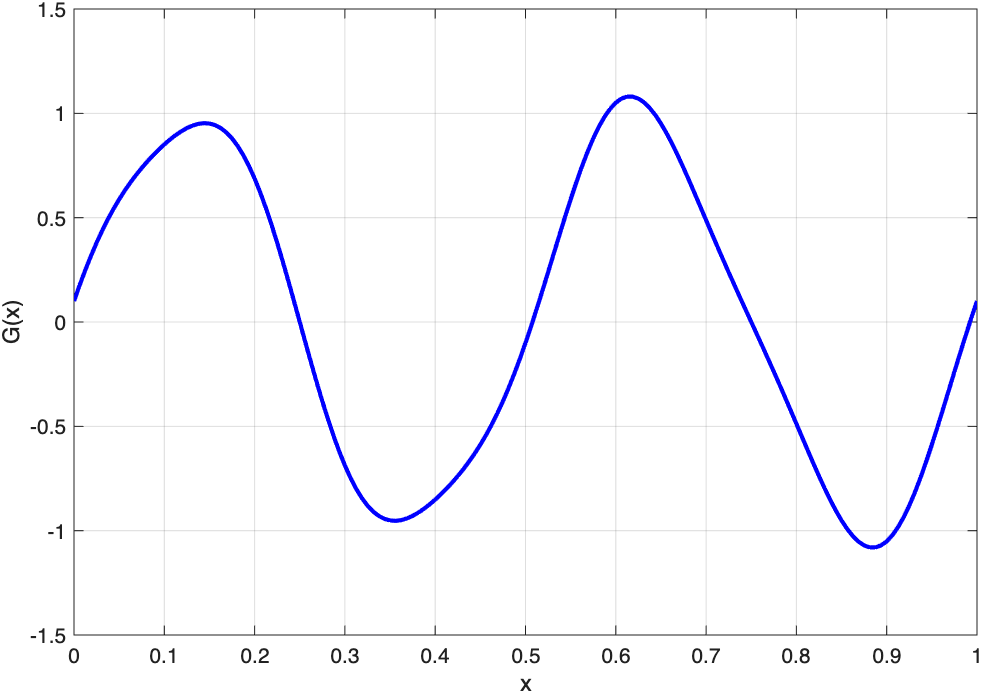}
        \caption{Terminal cost $G$}
        \label{fig:subr2}
    \end{subfigure}
    \begin{subfigure}[b]{0.32\textwidth}
        \includegraphics[width=\textwidth]{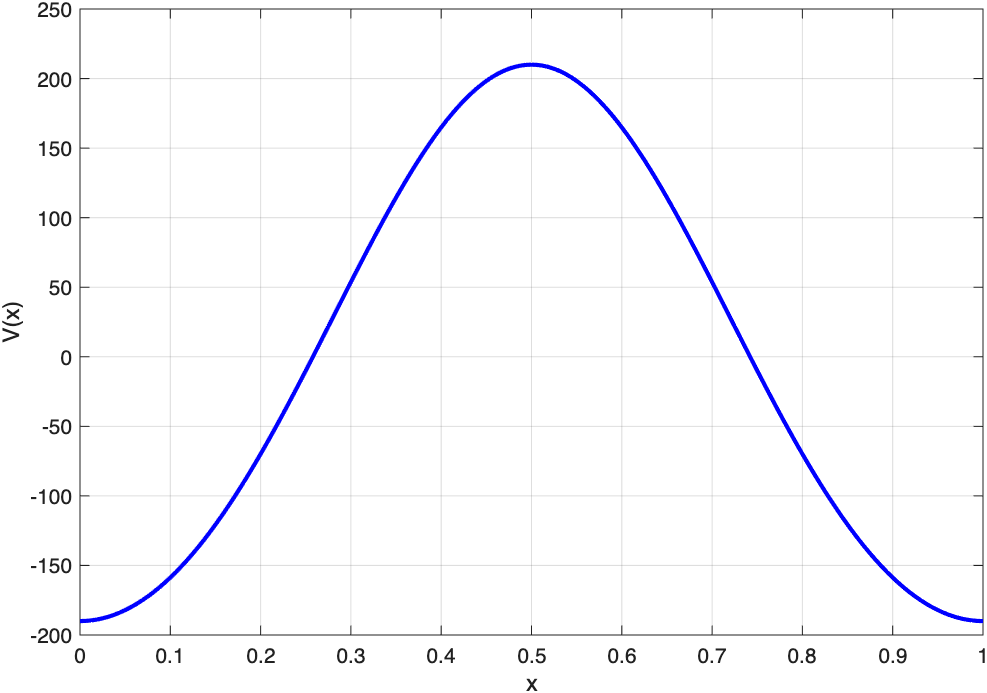}
        \caption{Potential $V$}
        \label{poteniall}
    \end{subfigure}
    \caption{Test 2. The initial data.}
    \label{data}
\end{figure}

\subsubsection{Performance Analysis of Linear Solvers}
At each Newton step, the linear system~\eqref{matrix-form} is solved using three different iterative solvers: Block Gauss--Seidel, GMRES, and BiCGSTAB. This comparison is performed for both the diffusive regime ($\nu = 0.4$) and the nearly deterministic setting ($\nu = 0.02$) to ensure the robustness of the solver selection.

Table~\ref{tab:solver_comparison} reports the corresponding iteration counts and CPU times for both cases. In the case $\nu = 0.4$, all three methods satisfy the stopping criterion in 6 iterations. However, as the diffusion parameter decreases to $\nu = 0.02$, the Krylov subspace methods (GMRES and BiCGSTAB) require significantly more iterations and computational time, whereas the Block Gauss--Seidel method remains highly efficient. Across both regimes, Block Gauss--Seidel yields the lowest CPU time. Consequently, the results presented in the remainder of this study correspond to the Block Gauss--Seidel solver.

\begin{table}[htb!]
\centering
\begin{tabular}{l c c c c c}
\hline
& \multicolumn{2}{c}{$\nu = 0.4$} & & \multicolumn{2}{c}{$\nu = 0.02$} \\
\cline{2-3} \cline{5-6}
Solver & Iterations & CPU (s) & & Iterations & CPU (s) \\
\hline
Block Gauss--Seidel & 6 & 0.054 & & 8  & 0.946 \\
GMRES               & 6 & 4.251  & & 13 & 5.914  \\
BiCGSTAB            & 6 & 1.832  & & 13 & 4.382  \\
\hline
\end{tabular}
\caption{Iteration counts and CPU time for the different solvers across both diffusion regimes ($\nu=0.4$ and $\nu=0.02$).}
\label{tab:solver_comparison}
\end{table}

\subsubsection{Test 2.1 with $\nu= 0.4$}\label{sec:test2_1}

We start with the case $\nu = 0.4$. Figure~\ref{fig:errors} shows, in logarithmic scale, the evolution of the Newton iteration errors $E(m^n)$ and $E(u^n)$ for the three schemes: Newton--SL, Newton--FD, and FD--Newton.
 All the three schemes exhibit comparable convergence behavior, whereas FD--Newton shows a slightly faster decay of the $\ell^\infty$--norm of successive iterates.

In Figure~\ref{fig:quadratic-convergence} we compare the convergence rate of the three
Newton algorithms by plotting $\log(m^{n+1})$ against $\log(m^{n})$. The estimated convergence rate is computed as the slope of the least-squares linear fit of
the data points $(\log(m^{n}),\log(m^{n+1}))$, which provides a numerical
estimate of the exponent $p$ in the relation $m^{n+1} \approx C (m^{n})^{p}$. 
This results in $p=1.408$ for Newton--SL, $p=1.393$ for Newton--FD and $p=1.663$ for FD--Newton.
For a quadratically exact convergent method we expect $p=2$, hence an ideal slope of
$2$ in this log--log plot.

Figure~\ref{fig:dist} displays the space--time evolution of the approximated distributions computed by the three schemes. 
It can be observed that, at the final time, the density concentrates around the points where the terminal cost $G$, shown in Figure~\ref{fig:subr2}, attains its minima. 
Finally, Figure~\ref{fig:vf} presents the corresponding space--time plots of the approximated value functions obtained with the same schemes. 
The solutions appear smooth, as expected for the relatively large diffusion coefficient $\nu = 0.4$. 
Overall, the three methods produce visually consistent results, in good agreement with the reference test reported in~\cite{litest}, confirming the reliability and coherence of the proposed discretizations.
 The near identity of the results for Newton-FD and FD-Newton can be explained by the fact that both strategies employ the same finite difference spatial operators, leading to similar truncation errors and numerical diffusion profiles.
The Newton-SL scheme behaves differently because it is based on a distinct discretization principle derived from the Feynman--Kac representation, which evolves the solution along characteristics. In FD schemes, numerical diffusion is mainly introduced by the upwind stencils used to stabilize the advection term, resulting in grid-based artificial dissipation. In contrast, in SL schemes the dominant source of diffusion arises from the interpolation step at the foot of the characteristics, yielding a different error distribution.
Moreover, Newton-SL is run with a smaller time step, $\Delta t = \mathcal{O}(h^{3/2})$, compared to the FD scheme, which uses $\Delta t = \mathcal{O}(h)$. This higher temporal resolution, combined with the interpolation-based dissipation mechanism, results in a slightly less diffusive approximation in the Newton-SL solution and explains the small variations observed in Figure~\ref{fig:dist}.
\begin{figure}[htb!]
    \centering
    \begin{subfigure}[b]{0.45\textwidth}
        \includegraphics[width=\textwidth]{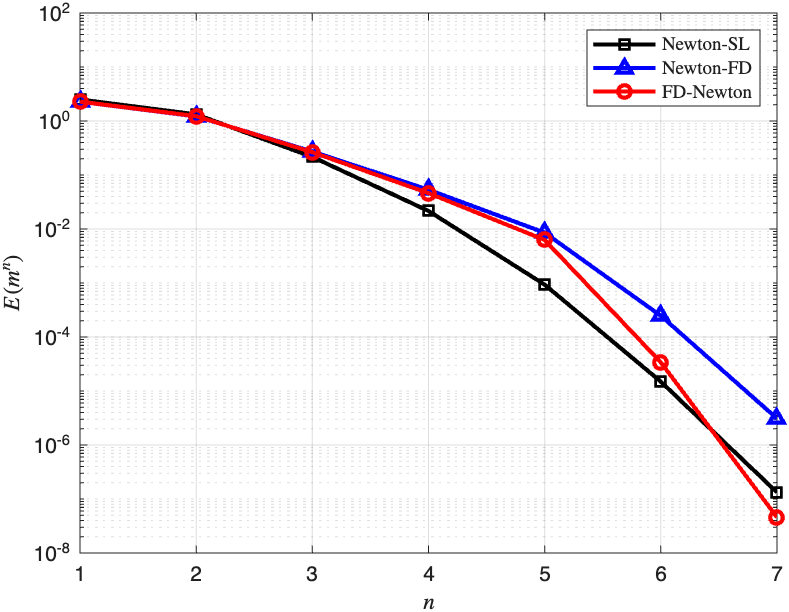}
        \caption{$E(m^n)$}
        \label{fig:sub1}
    \end{subfigure}
    \begin{subfigure}[b]{0.45\textwidth}
        \includegraphics[width=\textwidth]{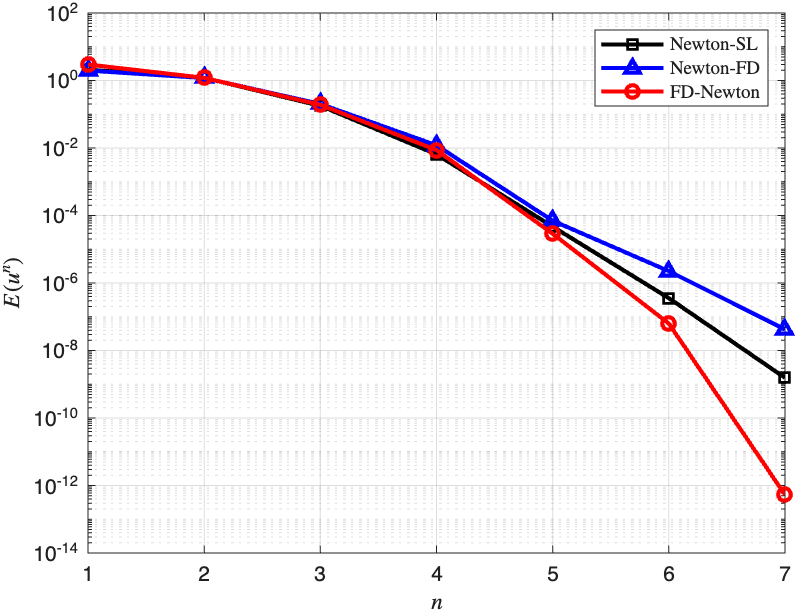}
        \caption{$E(u^n)$}
        \label{fig:sub2}
    \end{subfigure}
    \caption{Test 2.1. Convergence history of the Newton iterations for Newton--SL, Newton--FD, and FD--Newton.}
    \label{fig:errors}
\end{figure}

\begin{figure}[htb!]
    \centering
\includegraphics[width=0.5\linewidth]{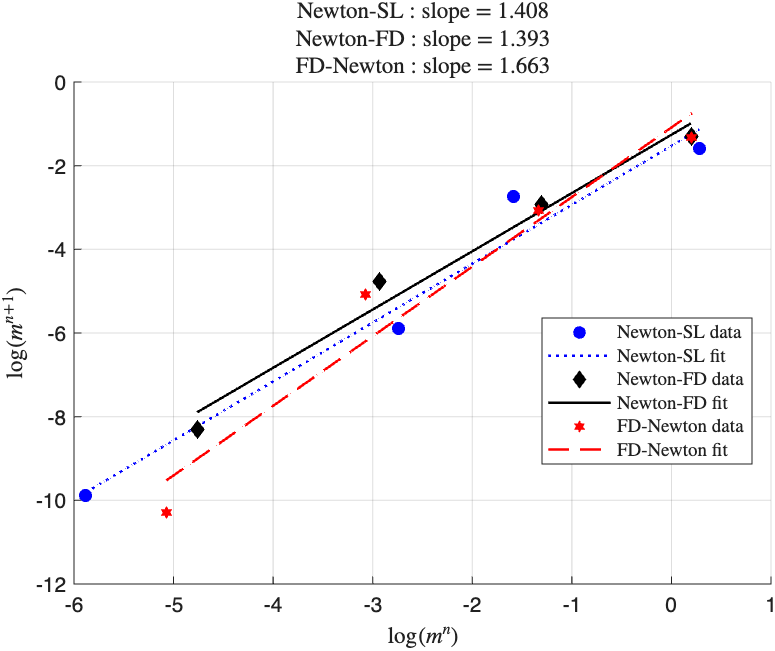}
    \caption{Log–log plot of $\log(m^n)$ versus $\log(m^{n+1})$ for the three Newton-type schemes (Newton–SL, Newton–FD and FD–Newton).}
    \label{fig:quadratic-convergence}
\end{figure}

\begin{figure}[htb!]
    \centering
    \begin{subfigure}[b]{0.32\textwidth}
        \includegraphics[width=\textwidth]{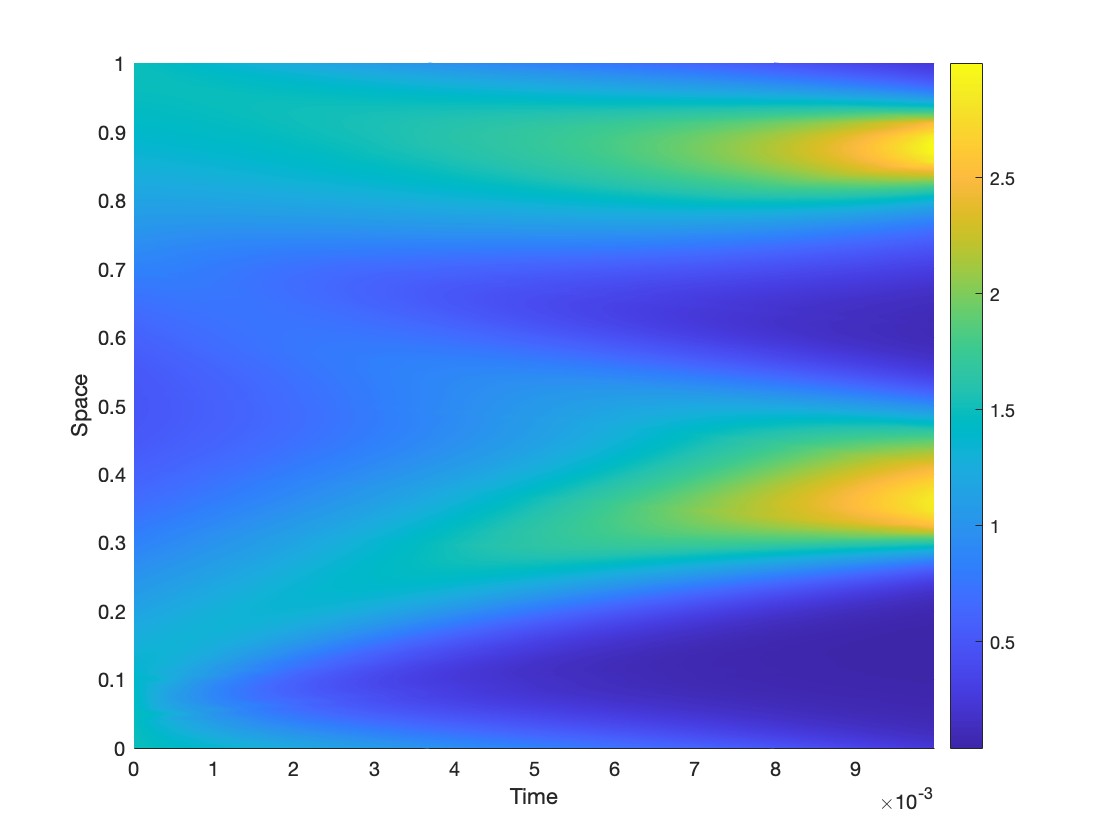}
        \caption{Newton-SL}
        \label{fig:sub1}
    \end{subfigure}
    \begin{subfigure}[b]{0.32\textwidth}
        \includegraphics[width=\textwidth]{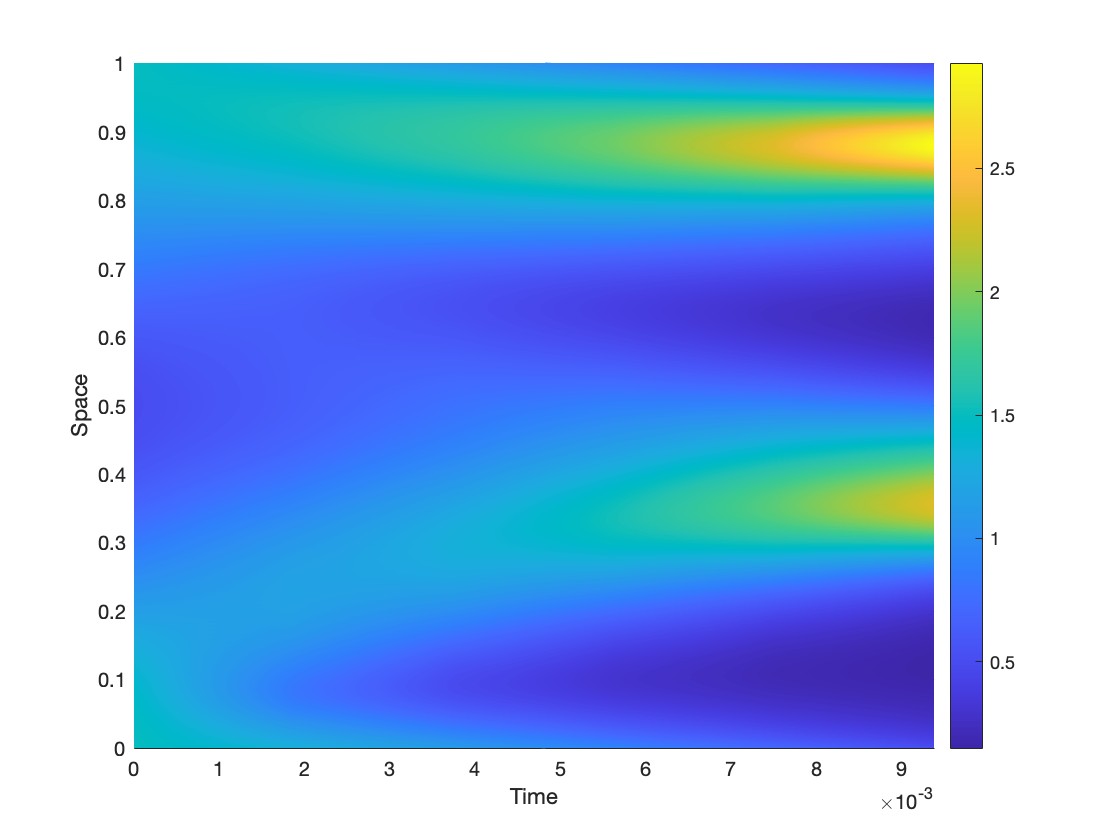}
        \caption{Newton-FD}
        \label{fig:sub2}
    \end{subfigure}
    \begin{subfigure}[b]{0.32\textwidth}
        \includegraphics[width=\textwidth]{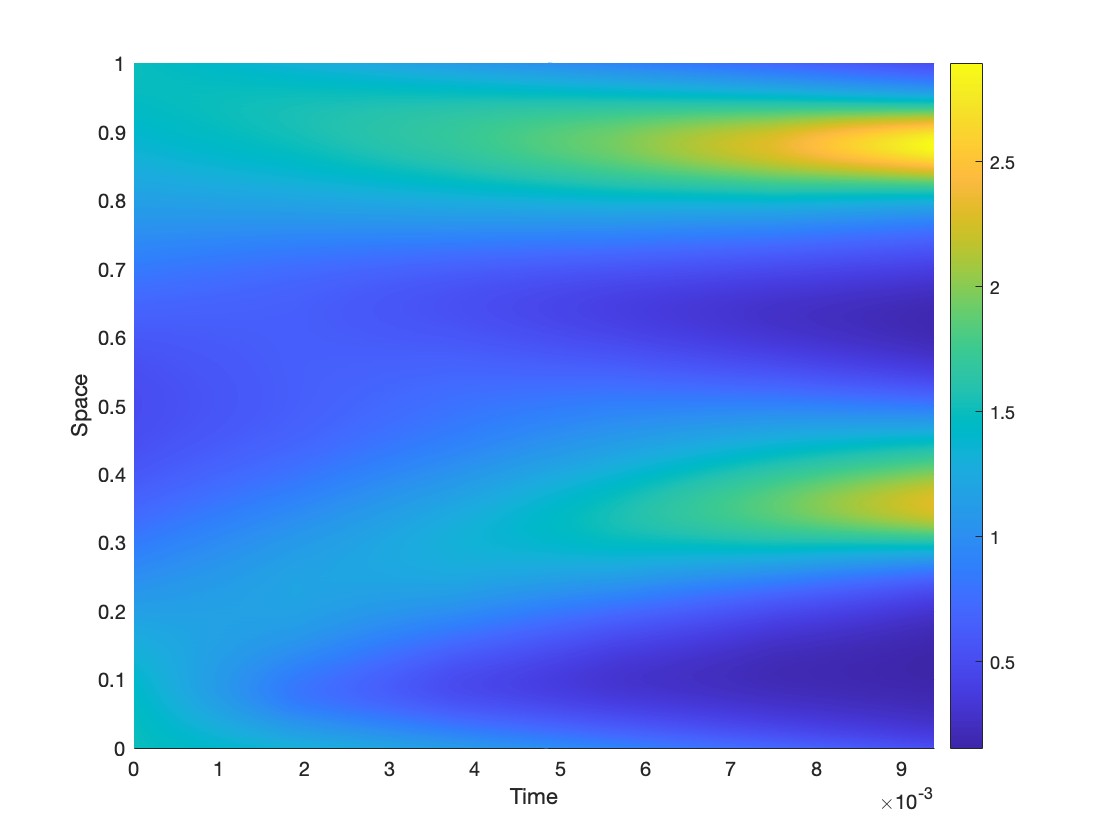}
        \caption{FD-Newton}
        \label{fig:sub3}
    \end{subfigure}
\caption{Test 2.1. Space--time evolution of the approximated distributions $m$ computed by Newton--SL, Newton--FD, and FD--Newton. 
    \label{fig:dist}}
\end{figure}

\begin{figure}[htb!]
    \centering
    \begin{subfigure}[b]{0.32\textwidth}
        \includegraphics[width=\textwidth]{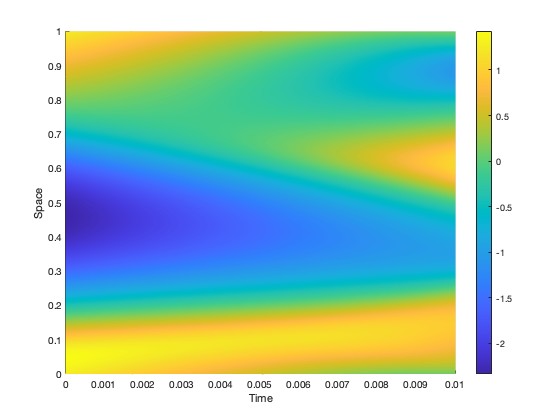}
        \caption{Newton-SL}
        \label{fig:sub1}
    \end{subfigure}
    \begin{subfigure}[b]{0.32\textwidth}
        \includegraphics[width=\textwidth]{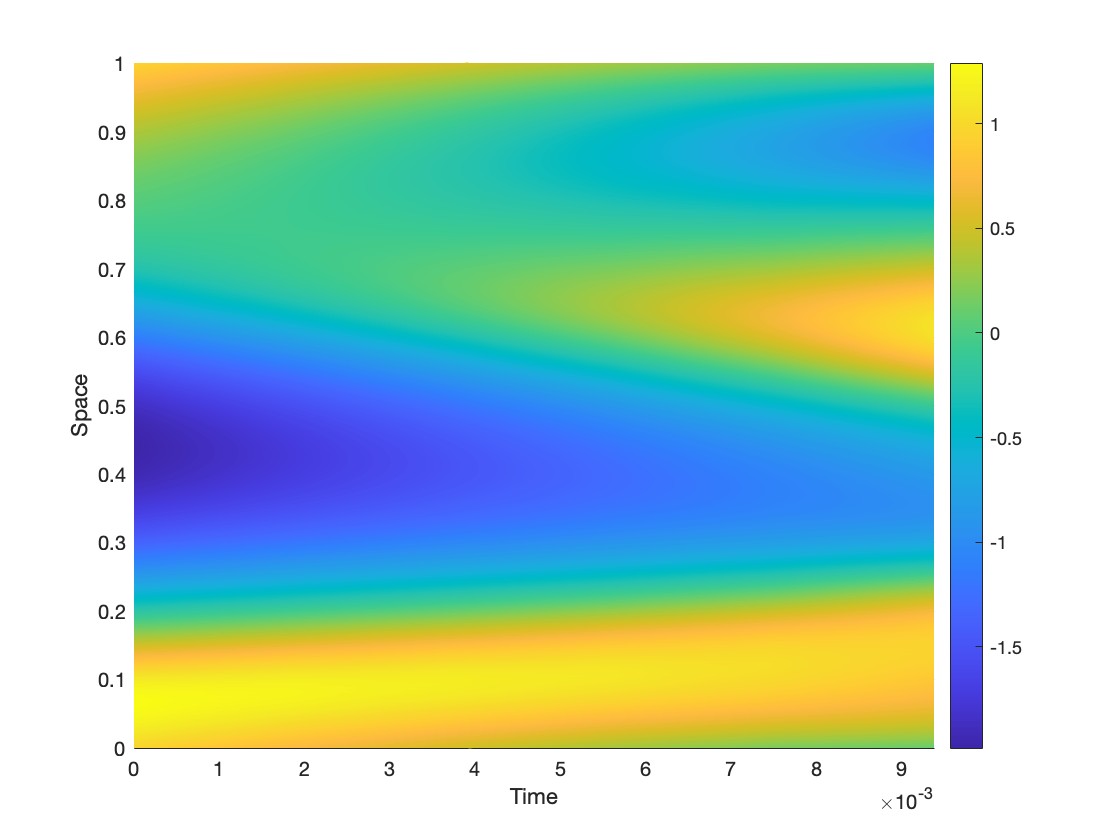}
        \caption{Newton-FD}
        \label{fig:sub2}
    \end{subfigure}
    \begin{subfigure}[b]{0.32\textwidth}
        \includegraphics[width=\textwidth]{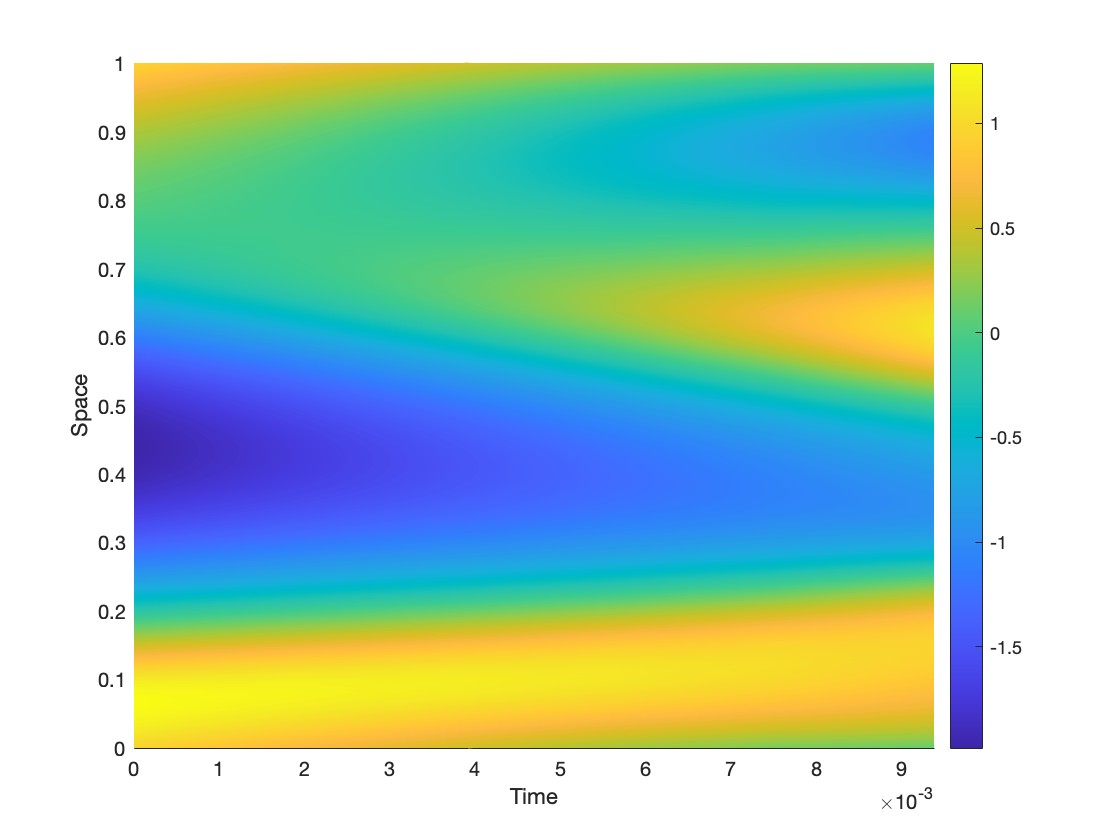}
        \caption{FD-Newton}
        \label{fig:sub3}
    \end{subfigure}
    \caption{Test 2.1. Space--time evolution of  approximated
    value function $u$ computed by Newton-SL, Newton-FD, and FD-Newton methods.}
    \label{fig:vf}
\end{figure}

\subsubsection{Test 2.2}\label{test2}
   We now consider a smaller diffusion coefficient, $\nu = 0.02$.    
As shown in Figure~\ref{fig:s3}, Newton--SL iterations converge after seven steps, reaching the prescribed tolerance. 
In contrast, both Newton--FD and FD--Newton iterations experience breakdown after only a few iterations, as the computed distribution becomes negative. 
This behavior highlights the increased robustness of the Newton--SL scheme in regimes characterized by low diffusion. 
Figures~\ref{fig:s2}--\ref{fig:sb3} display the space--time evolution of the approximated distribution and the corresponding value function obtained with Newton-SL.

Since both FD--Newton and Newton--FD fail when applying local iterations, we employ the global Newton method described in Section~\ref{sec:GN}, 
using the parameters $\beta = 1/2$ and $c = 1/3$ for the line-search procedure. 
These values  ensures sufficient decrease of the merit function and robust globalization of the iterations. 
As expected, both schemes converge under this approach. 
  Figure~\ref{FDglobal3} shows the Newton residual error, the merit function $\Theta$ \eqref{merit_f}, and the Armijo step size $\alpha_n$ along the globalized Newton iterations for the Newton--FD scheme.
Figure~\ref{FDglobal} displays the time evolution of the approximated distribution and value function computed by the global Newton--FD scheme (at convergence).
For brevity, we omit the analogous results for FD--Newton, whose performance is comparable to that of Newton--FD.
These results confirm the effectiveness of the globalized Newton strategy in stabilizing the discretizations and achieving convergence even in challenging small-diffusion MFG regimes.

\begin{remark}
In \cite{Achlau}, the authors solve a finite difference discretization of the MFG system by employing Newton's method combined with a continuation strategy with respect to the diffusion parameter $\nu$ (see also \cite{achdou1,achdou2}). 
This approach is particularly effective for handling problems with small diffusion values. 
Specifically, the problem is first solved for a large value of $\nu$, and the corresponding solution is then used as the initial guess to solve, still by Newton's method, the discrete MFG system with a smaller diffusion coefficient. 
The procedure is repeated until the target (small) viscosity is reached. 
We point out that such a continuation approach has a similar effect to the global Newton iterations, as both techniques aim to improve the robustness of convergence when the local Newton method fails due to poor initialization.
\end{remark}

\begin{figure}[htb!]
\begin{subfigure}[b]{0.4\textwidth}
        \includegraphics[width=\textwidth]{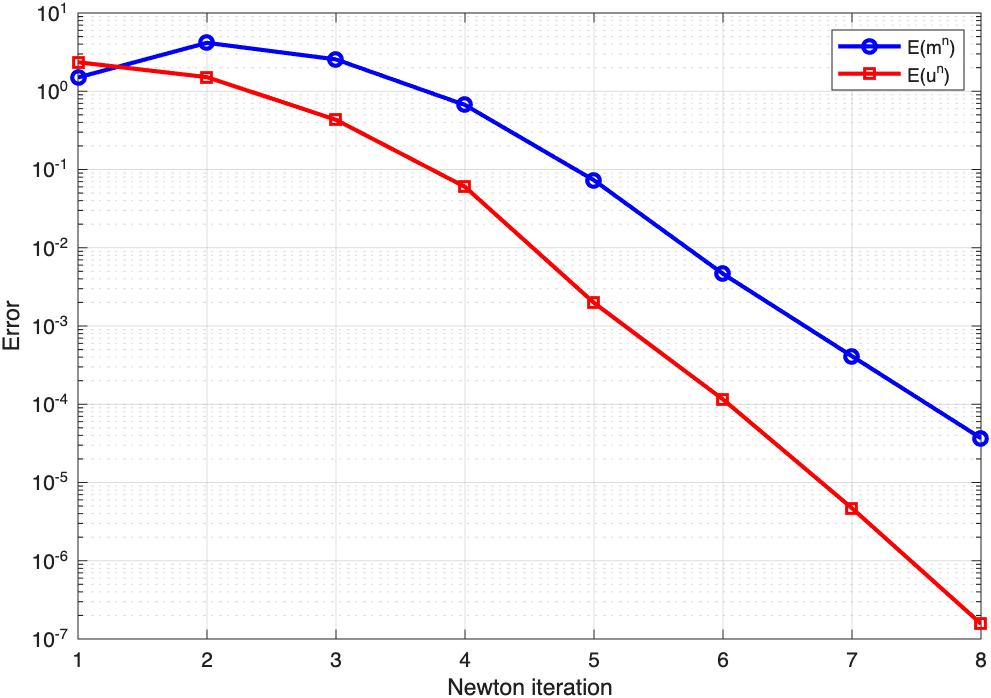}
        \caption{Errors $E(m^n)$ and $E(u^n)$}
        \label{fig:s3}
    \end{subfigure}
    \\
    \begin{subfigure}[b]{0.4\textwidth}
        \includegraphics[width=\textwidth]{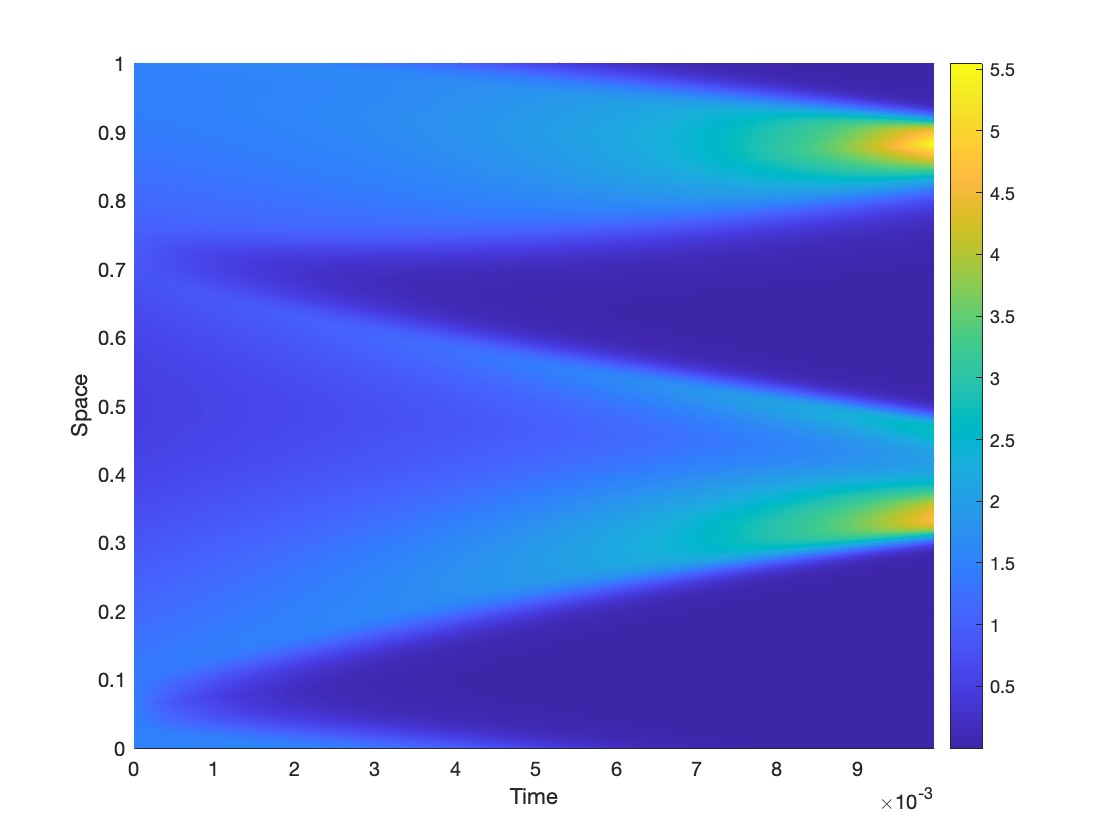}
        \caption{Evolution of $m$}
        \label{fig:s2}
    \end{subfigure}
    \begin{subfigure}[b]{0.4\textwidth}
        \includegraphics[width=\textwidth]{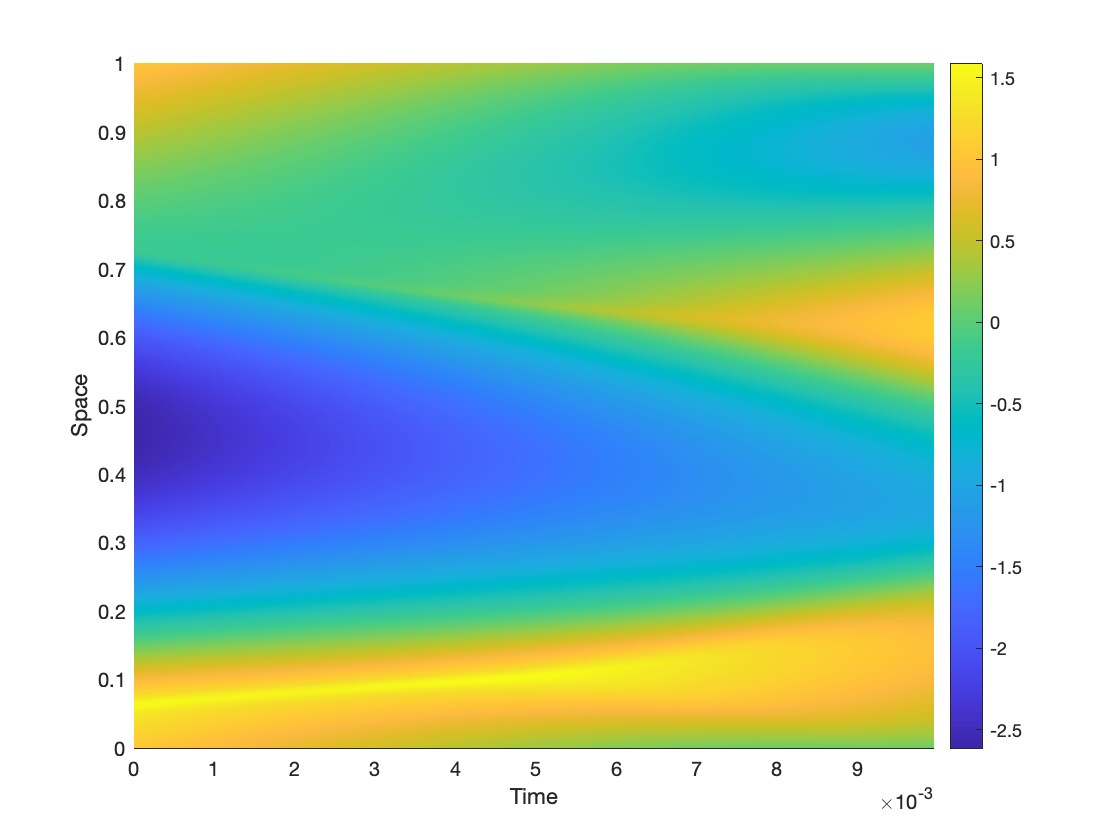}
        \caption{Evolution of $u$}
        \label{fig:sb3}
    \end{subfigure}

    \caption{
    Convergence history of the Newton–SL iterations and space--time evolution of the approximated distribution $m$ and value function $u$ for $\nu = 0.02$.
     }
    \label{sigma_0.02..}
\end{figure}

\begin{figure}[htb!]
    \centering
    \begin{subfigure}[b]{0.32\textwidth}
        \includegraphics[width=\textwidth]{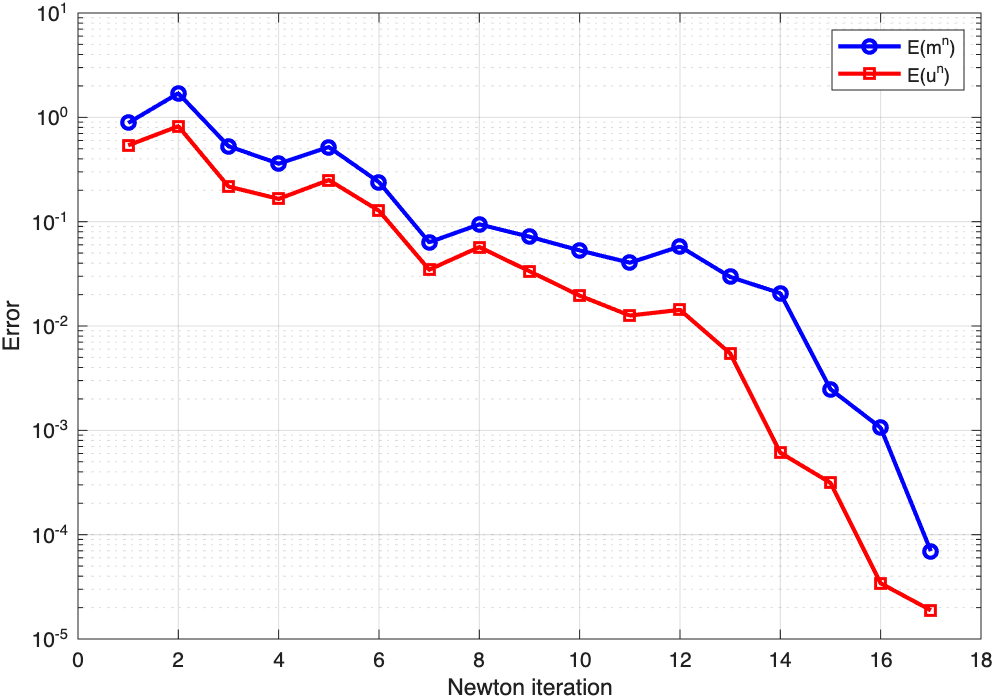}
        \caption{Errors $E(m^n)$ and $E(u^n)$}
    \label{fig:sub2}
    \end{subfigure}
    \begin{subfigure}[b]{0.31\textwidth}
        \includegraphics[width=\textwidth]{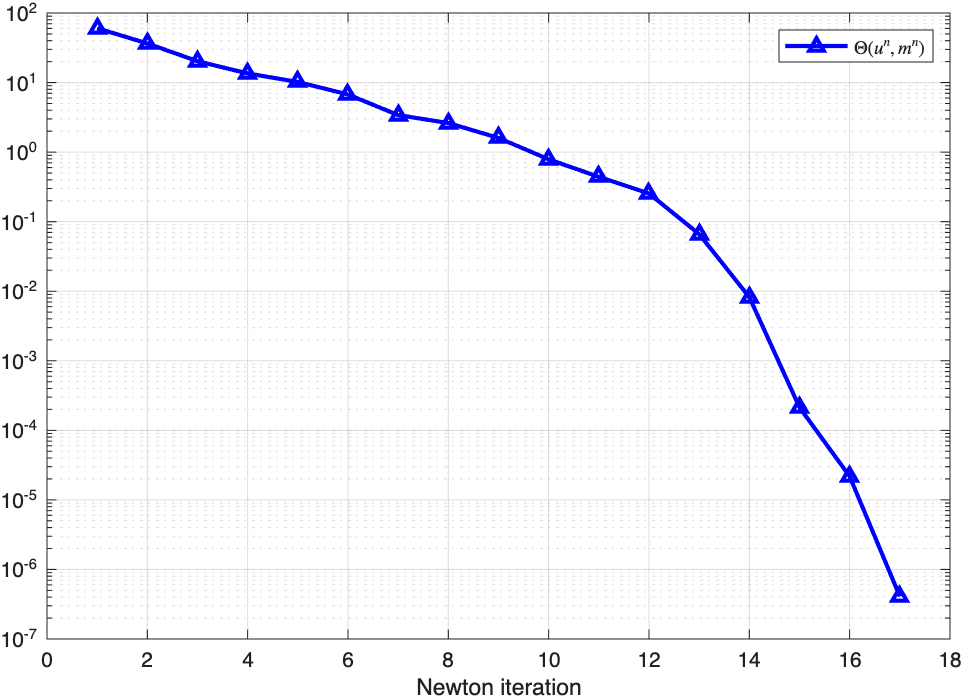}
        \caption{Merit function $\Theta(u^n, m^n)$}
        \label{fig:sub2}
    \end{subfigure}
    \begin{subfigure}[b]{0.31\textwidth}
        \includegraphics[width=\textwidth]{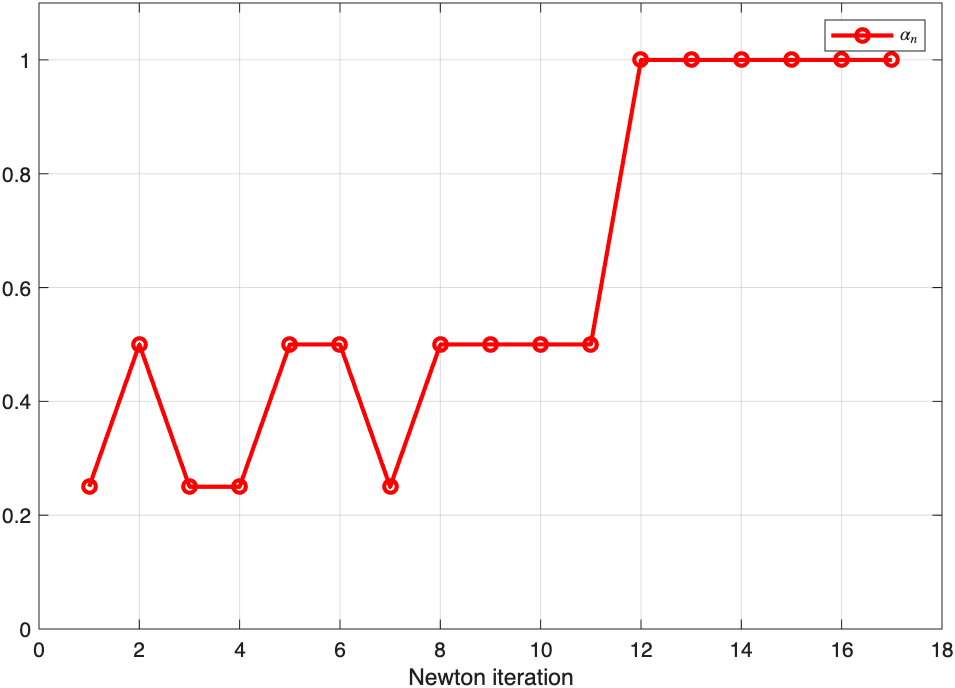}
        \caption{Step size $\alpha^n$}
        \label{fig:sub33}
    \end{subfigure}
    \caption{Test 2.2. 
From left to right: Convergence history for the global Newton iterations for Newton--FD with $\nu=0.02$, the merit function 
$\Theta(u^n, m^n)$, 
and the Armijo step size $\alpha^n$.}
      \label{FDglobal3}
\end{figure}

\begin{figure}[htb!]
    \begin{subfigure}[b]{0.4\textwidth}
        \includegraphics[width=\textwidth]{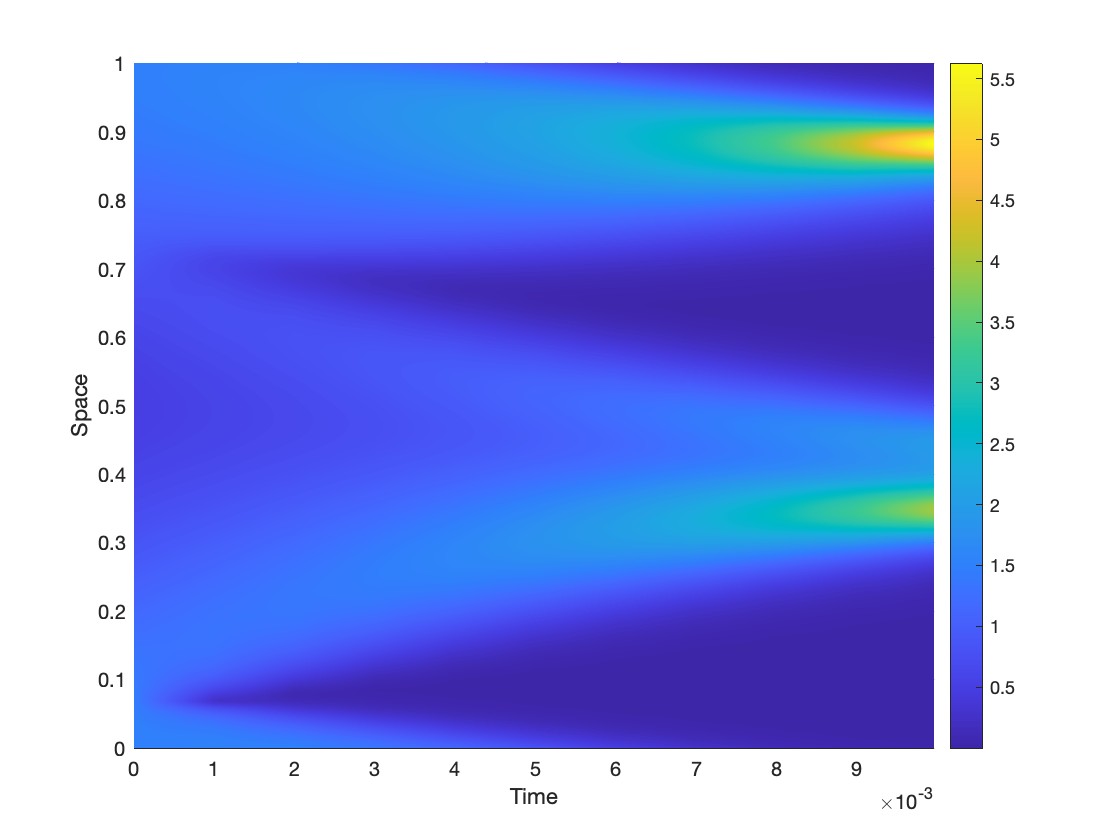}
        \caption{Evaluation of $m$}
        \label{fig:sub2}
    \end{subfigure}
    \begin{subfigure}[b]{0.4\textwidth}
        \includegraphics[width=\textwidth]{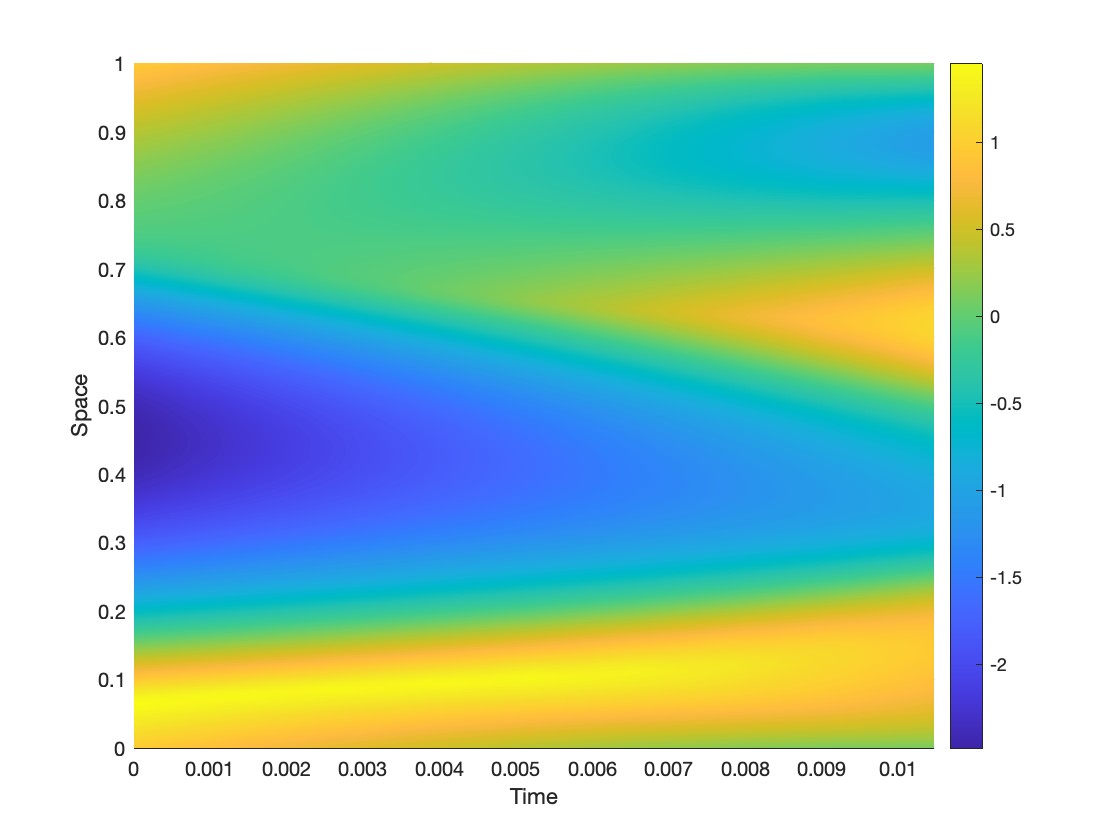}
        \caption{Evaluation of $u$}
        \label{fig:sub3}
    \end{subfigure}
    \caption{Test 2.2.
    Space--time evolution of the approximated distribution $m$ and value function $u$ by global Newton-FD for $\nu = 0.02$.}
      \label{FDglobal}
\end{figure}

\begin{remark}
    A simple alternative to the globalized Newton method is the relaxed Newton method,
    which uses a fixed damping parameter $\alpha \in (0,1)$ at every iteration.
    The globalized Newton method with Armijo line search can be viewed as an adaptive
    generalization: it automatically selects the damping parameter $\alpha_n$ at each
    iteration. As shown in Figure~\ref{FDglobal3}, the step size is small in the early
    iterations where the initial guess is far from the solution, and increases to
    $\alpha_n = 1$ as the method approaches convergence. This shows that damping is
    necessary in the early phase, but becomes unnecessary near the solution, a behavior
    that a fixed $\alpha < 1$ cannot capture.
\end{remark}

\subsection{Test 3: a MFG problem with non separable Hamiltonian}

We consider a one-dimensional MFG system featuring a non-separable Hamiltonian,
\begin{equation*}
\begin{cases}
    & -\partial_t u - 0.05\Delta {u} + \frac{1}{2(1+4m)^\gamma}\lvert Du\rvert^2 = \zeta m \quad \text{in } [0,1]\times[0,1], \\
    & \partial_t m - 0.05\Delta m - \text{div}\left(\frac{mDu}{(1+4m)^\gamma}\right) = 0 \quad \text{in }  [0,1]\times[0,1],\\
    &G(x)=10\min \{(x-0.3)^2,\,(x-0.7)^2\}\quad\text{in }  [0,1],\\
    &m_0(x)=4\;\text{for }x\in[0.375,0.625],\;m_0(x)=0 \; \text{otherwise}
    \end{cases}
\end{equation*}
In this test, we compute the solution of the system using the Newton--SL scheme, 
with parameters $\gamma = 1.5$, $\zeta = 1$, spatial step $h = 5 \cdot 10^{-3}$, 
and time step $\Delta t = h^{3/2}/2$.

In Figure~\ref{non-sep-1d}, we display the time evolution of the approximated density and value function at the time instants $t_k \in \{0, 0.2, 0.4, 0.6, 0.8, 1\}$. 
We observe that the population distribution splits into two groups: one moving toward the left target and the other toward the right target. 
However, due to the presence of congestion and crowd-aversion effects, neither group fully concentrates on its respective target at the final time. For this non-separable Hamiltonian test, the simulation obtained with the Newton--SL scheme 
is stable and accurate, and remains fully consistent with the reference outcomes reported in~\cite{litest}.

We then vary the value of the diffusion coefficient $\nu$ to investigate its influence on the solution.
Figure~\ref{diffrent-diff} reports the terminal density $m(T,x)$ and the initial value function $u(0,x)$
for $\nu \in \{0.2,\,0.05,\,0.005\}$.
In particular, we can see that decreasing $\nu$ reduces diffusive smoothing
and yields a more pronounced separation/concentration of the distribution at final time. Finally, Table~\ref{tab:test3} reports the corresponding iteration counts and CPU times for each value of $\nu$.

\begin{figure}[htb!]
  \centering
\includegraphics[width=7.5cm]{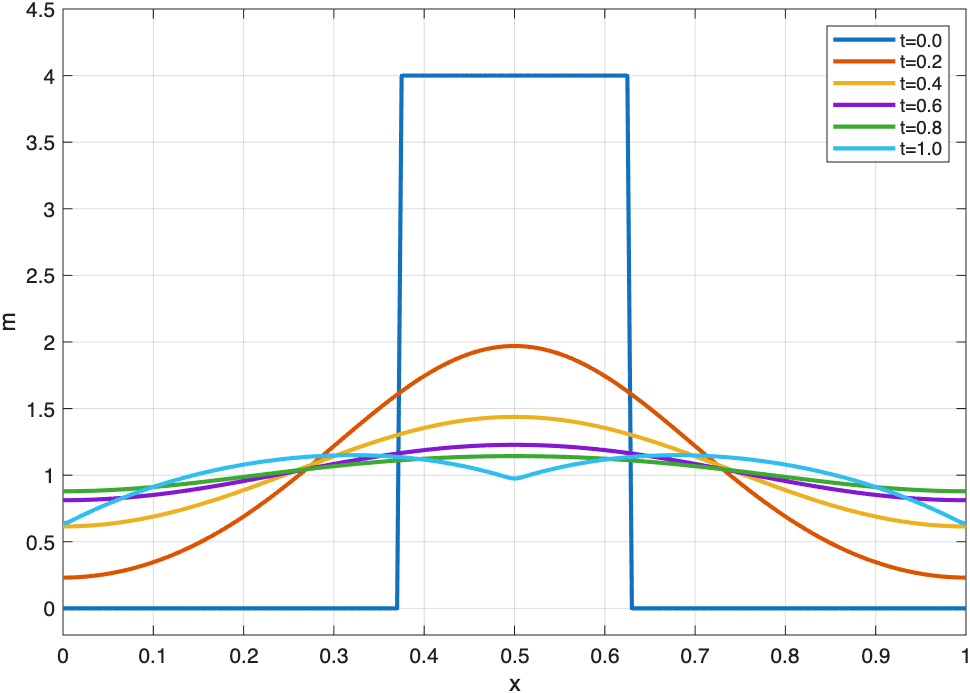}
    \includegraphics[width=7.5cm]{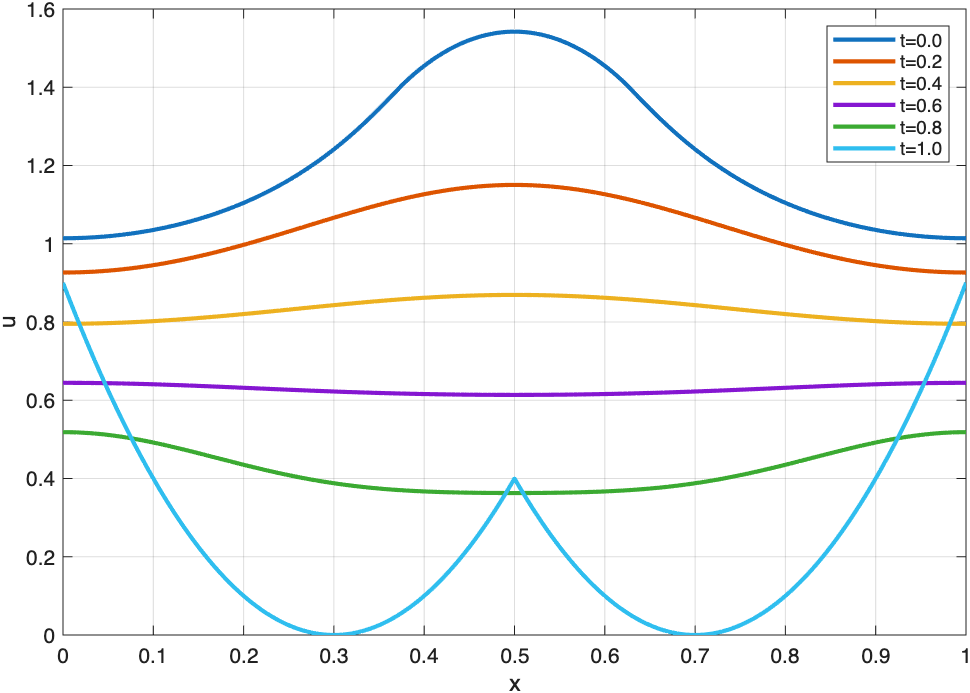}
   \caption{Test~3. Time evolution of the approximated density (left) and approximated value function (right) }
 \label{non-sep-1d}
\end{figure}

\begin{figure}[htb!]
  \centering
\includegraphics[width=7.5cm]{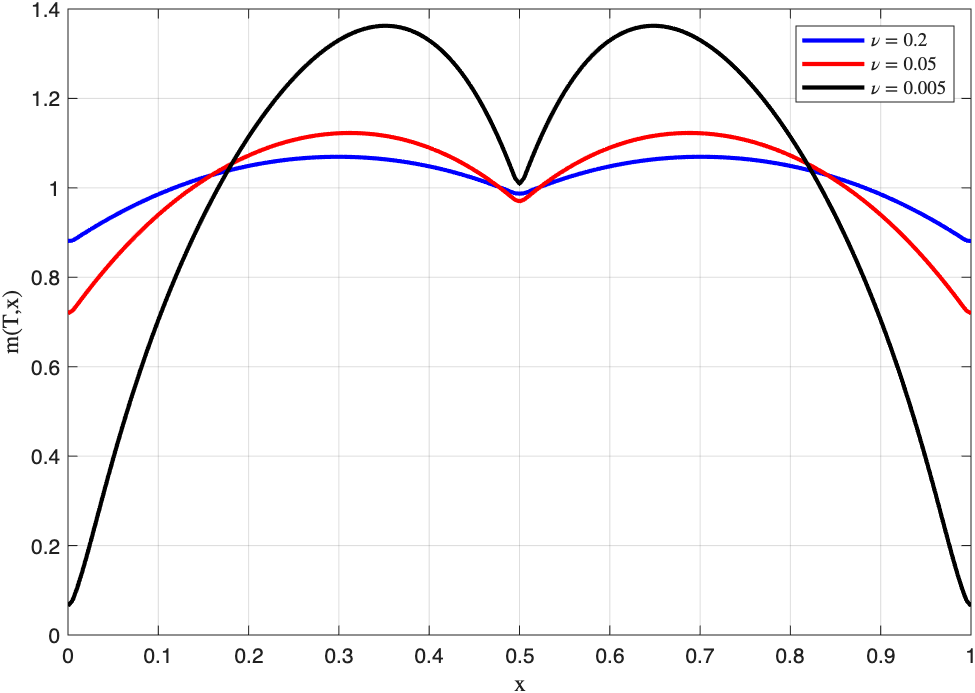}
    \includegraphics[width=7.5cm]{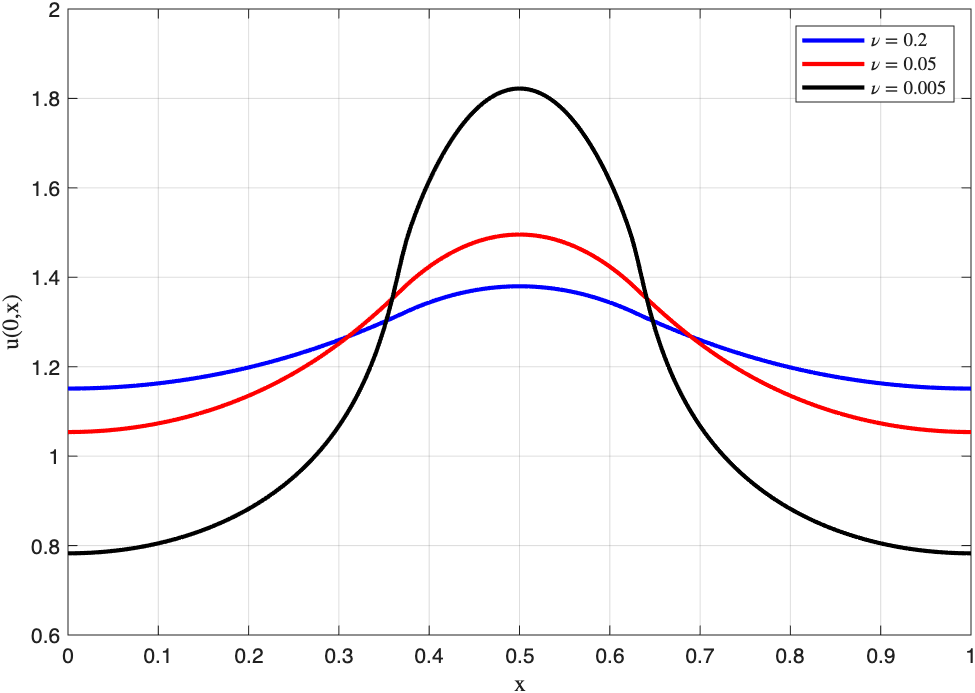}
   \caption{Test~3. The approximated density at final time (left) and the approximated value function at initial time (right) for different values of $\nu$ }
 \label{diffrent-diff}
\end{figure}

\begin{table}[htb!]
\centering
\begin{tabular}{lcc}
\hline
$\nu$ & Number of iterations & CPU (s) \\
\hline
0.2      & 5 & 2.631 \\
0.05   & 6 & 4.064 \\
0.005& 14 & 12.83 \\
\hline
\end{tabular}
\caption{Iteration counts and CPU time for different diffusion coefficients $\nu$.}
\label{tab:test3}
\end{table}

\subsection{Test 4: a MFG with separable Hamiltonian in dimension two}
We now consider a two-dimensional MFG system, previously investigated numerically in~\cite{litest}. 
The computational domain is $[0,1] \times (0,1)^2$, with diffusion coefficient $\nu = 0.4$, 
and the following data:
\begin{align*}
        m_0(x_1,x_2)&=1+12\cos(2\pi x_1)+12\cos(2\pi x_2),\\
          G(x_1,x_2)&=\cos(2\pi x_1)+\cos(2\pi x_2),\\
          H(x_1,x_2,p)&=|p|^2+\sin(2\pi x_1)+\sin(2\pi x_2)+\cos(4\pi x_1),\\
        F(m)&=m^2.
\end{align*}
We set $h = 10^{-2}$ and $\Delta t = h^{3/2}/2$. 
The evolution of the approximated distribution $m$, computed with the Newton--SL scheme, 
is displayed in Figure~\ref{dist_m} at times $t_k \in \{0, 0.5, 0.75, 1\}$. 
Figures~\ref{fig:sub2}--\ref{fig:sub3} illustrate the stationary profile attained by the density at intermediate times, 
which can be interpreted as a manifestation of the turnpike effect.
Turnpike phenomena for mean field games with local couplings have been discussed in~\cite{turnpike}.
The results are consistent with those reported in~\cite{litest}, confirming the accuracy of the proposed scheme 
in reproducing two-dimensional MFG dynamics.

\begin{figure}[htb!] 
    \centering
    \begin{subfigure}{0.45\textwidth}
        \includegraphics[width=\linewidth]{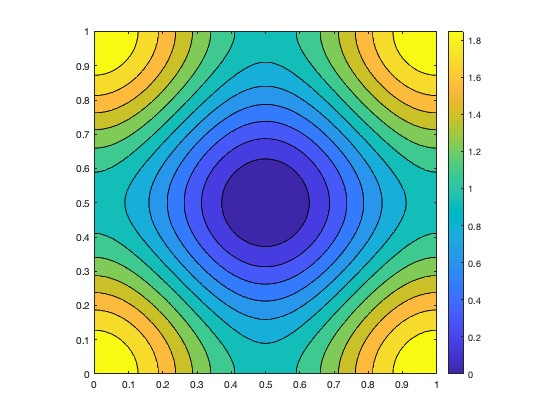}
        \caption{$t_k=0$}
        \label{fig:sub1}
    \end{subfigure}
    \begin{subfigure}{0.45\textwidth}
        \includegraphics[width=\linewidth]{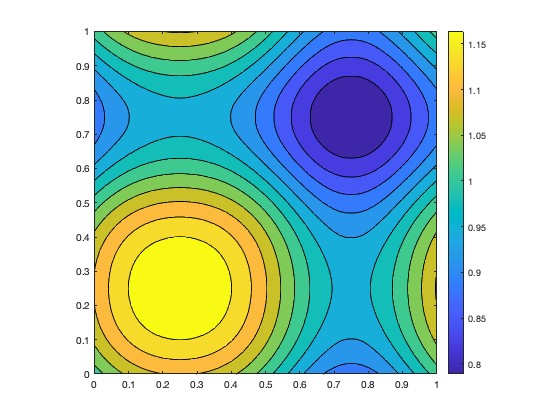}
        \caption{$t_k=N_t/2$}
        \label{fig:sub2}
    \end{subfigure}
    \begin{subfigure}{0.45\textwidth}
        \includegraphics[width=\linewidth]{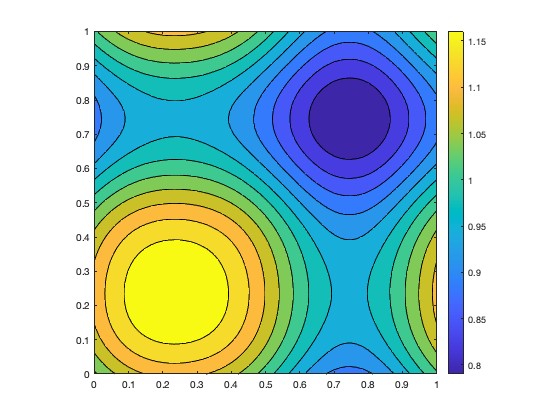}
        \caption{$t_k=3N_t/4$}
        \label{fig:sub3}
    \end{subfigure}
    \begin{subfigure}{0.45\textwidth}
        \includegraphics[width=\linewidth]{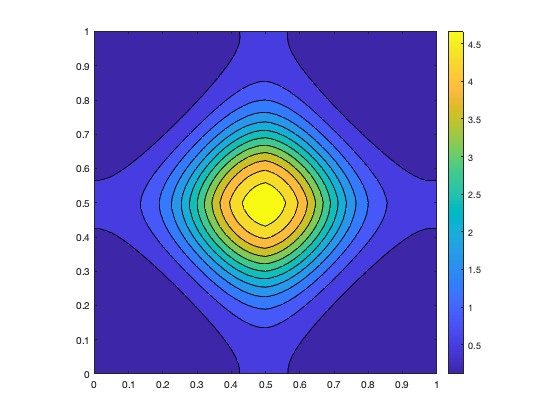}
        \caption{$t_k=N_t$}
        \label{fig:sub33}
    \end{subfigure}
    \caption{
    Test~4. Time evolution of the approximated distribution $m$ at with Newton-SL at different times.}
    \label{dist_m}
\end{figure}

\section{Conclusion}
In this work, we have proposed a new technique for numerically solving a second-order mean field games system with local coupling by applying Newton iterations in infinite dimensions. Our study is based on a series of numerical experiments comparing two possible strategies: performing Newton iterations at the continuous level and then discretizing, or discretizing first and applying Newton iterations to the resulting finite-dimensional non linear system.
A first remark is that the former strategy leads to a relatively simple scheme, since each Newton step reduces to approximate a system of two linear parabolic PDEs.
Both Newton--SL and Newton--FD follow a “iteration then discretization’’ strategy: the Newton direction is computed after discretization, and therefore the iterations do not follow the exact gradient of the continuous problem. As a consequence, the ideal quadratic convergence predicted at the continuous level is not fully observed in practice. In contrast, a “discretization then iteration” strategy, such as FD--Newton, builds the Newton iterations directly from the discrete problem and, in our tests, tends to follow more closely the true discrete gradient, sometimes displaying a behaviour closer to quadratic convergence.
Nevertheless, one should not expect the “discretization then iteration’’ strategy to systematically yield a globally smaller numerical error. Discretizing before or after the Newton step introduces different truncation errors, which may counterbalance any local improvement in the convergence rate.
To improve the performance of Newton-based schemes for MFG systems,
a promising direction is to employ higher-order discretization schemes within an iterate-then-discretize strategy
(e.g., high-order semi-Lagrangian or high-order finite difference methods).
This allows the numerical solution to remain closer to the continuous formulation,
thereby significantly enhancing both the convergence speed and the overall accuracy of the solver.
Finally, as expected, the Newton--SL scheme exhibits strong robustness in advection-dominated regimes, whereas both Newton--FD and FD--Newton require a stabilization strategy to prevent breakdown. In this regard, the use of a globalized Newton method proved particularly effective, especially for problems with small viscosity coefficients.

\section*{Acknowledgments}
The first author were partially supported by Italian Ministry of Instruction, University and Research (MIUR) (PRIN Project2022238YY5, ``Optimal control problems: analysis, approximation'') and by INdAM–GNCS Project, codice CUP$\_$E53C24001950001. 

\bibliographystyle{plain}
\bibliography{references}
\end{document}